\documentclass[12pt]{article}

\usepackage[english]{babel}
\usepackage[utf8]{inputenc}
\usepackage[top=1.5cm, bottom=2.5cm, left=2.5cm, right=2.5cm]{geometry}
\usepackage{mathrsfs}
\usepackage{amsmath}
\usepackage{amsthm}
\usepackage{amsfonts}
\usepackage{float}
\usepackage{color}
\usepackage{listings}
\usepackage{graphicx}

\title{\textbf{Universal aspects of critical percolation on random half-planar maps}}

\author{Loïc Richier%
  \thanks{UMPA, \'Ecole Normale Supérieure de Lyon, 46 allée d'Italie,
69364 Lyon Cedex 07, France. \newline Email: \texttt{loic.richier@ens-lyon.fr}}}

\date{\today}

\newtheorem{Th}{\bf Theorem}[section]
\newtheorem*{Thbis}{\bf Theorem}
\newtheorem{Def}[Th]{\bf Definition}

\newtheorem{Prop}[Th]{\bf Proposition}
\newtheorem*{Propbis}{\bf Proposition}
\newtheorem{Lem}[Th]{\bf Lemma}

\theoremstyle{definition}
 
\newtheorem{Alg}{\bf ALGORITHM}
\newtheorem*{Rk}{\bf Remark}
\newtheorem*{Rks}{\bf Remarks}
\newtheorem*{Ack}{\bf Acknowledgements}

\begin{document}

\maketitle

\begin{abstract}
We study a large class of Bernoulli percolation models on random lattices of the half-plane, obtained as local limits of uniform planar triangulations or quadrangulations. We first compute the exact value of the site percolation threshold in the quadrangular case using the so-called peeling techniques. Then, we generalize a result of Angel about the scaling limit of crossing probabilities, that are a natural analogue to Cardy's formula in (non-random) plane lattices. Our main result is that those probabilities are universal, in the sense that they do not depend on the percolation model neither on the degree of the faces of the map.
\end{abstract}

\section{Introduction}\label{SectionIntro}

In this work, we consider several aspects of Bernoulli percolation models (site, bond and face percolation) on Uniform Infinite Half-Planar Maps, more precisely on the Uniform Infinite Planar Triangulation and Quadrangulation of the half-plane (UIHPT and UIHPQ in short), which are defined as the so-called \textit{local limit} of random planar maps. Those maps, or rather their infinite equivalents (UIPT and UIPQ), were first introduced by Angel \& Schramm (\cite{angel_uniform_2003}) in the case of triangulations and by Krikun (\cite{krikun_local_2005}) in the case of quadrangulations (see also \cite{chassaing_local_2006}, \cite{menard_two_2010} and \cite{curien_view_2013} for an alternative approach). Angel later defined in \cite{angel_scaling_2004} the half-plane models, which have nicer properties. The Bernoulli percolation models on these maps are defined as follows: conditionally on the map, every site (respectively edge, face) is open with probability $p$ and closed otherwise, independently of every other sites (respectively edges, faces). All the details concerning the map and percolation models are postponed to Section \ref{SectionRPM}. More specifically, we will focus on the site percolation threshold for quadrangulations and the scaling limits of crossing probabilities.

\vspace{5mm}

In Section \ref{SectionPCQ}, we compute the site percolation threshold on the UIHPQ, denoted by $p_{c,\rm{site}}^{\square}$. This problem was left open in \cite{angel_percolations_2015}, where percolation thresholds are given for any percolation and map model (site, bond and face percolation on the UIHPT and UIHPQ), except for the site percolation on the UIHPQ (see also \cite{angel_growth_2003} and \cite{menard_percolation_2014}, where percolation thresholds are studied in the full-plane models). Roughly speaking, it is made harder by the fact that standard exploration processes of the map do not provide relevant information concerning percolation, as we will discuss later. This value is also useful in order to study the problem of the last section in the special case of site percolation on the UIHPQ. Namely, we will prove the following.

\begin{Th}\label{TheoremPCQ}
For Bernoulli site percolation on the UIHPQ, we have

$$p_{c,\rm{site}}^{\square}=\frac{5}{9}.$$Moreover, there is no percolation at the critical point almost surely.
\end{Th}

Due to the fact that Uniform Infinite Half-Planar Maps have a boundary, it is natural to consider boundary conditions. Here, the result holds for a free boundary condition. We believe that the percolation threshold is independent of these boundary conditions but did not investigate this further. We will discuss this again in Section \ref{SectionPCQ}.

\vspace{5mm}

The last section focuses on percolation models on Uniform Infinite Half-Planar Maps at their critical point, more precisely on crossing events. We work conditionally on the boundary condition of Figure \ref{InitialColouring} (in the case of bond percolation), where black edges are open, $a,b$ are fixed and positive and $\lambda$ is a positive scaling parameter.

\begin{figure}[h]
\begin{center}
\includegraphics[scale=1.6]{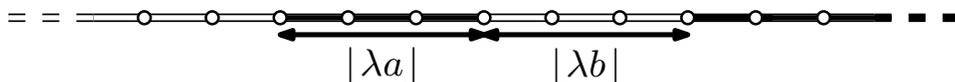}
\end{center}
\caption{The boundary condition for the crossing probabilities problem.}
\label{InitialColouring}
\end{figure}

Starting from now, we use the notation of \cite{angel_percolations_2015} and let $*$ denote either of the symbols $\triangle$ or $\square$, where $\triangle$ stands for triangulations and $\square$ for quadrangulations. The so-called \textit{crossing event} we are interested in is the event that the two black segments are part of the same percolation cluster. We denote by $\mathsf{C}^*(\lambda a,\lambda b)$ this event and study the \textit{scaling limit} of its probability when $\lambda$ goes to infinity. The first result was proved by Angel in the case of site percolation on triangulations.

\begin{Thbis}{(Theorem 3.3 in \cite{angel_scaling_2004})} Let $a,b\geq 0$. For site percolation at the critical point on the UIHPT,

$$\lim_{\lambda \rightarrow +\infty}{\mathbb{P}\left(\mathsf{C}_{\rm{site}}^{\triangle}(\lambda a,\lambda b)\right)} =\frac{1}{\pi}\arccos\left(\frac{b-a}{a+b}\right).$$

\end{Thbis} 

Our motivation is that this problem is a natural analogue to the famous \textit{Cardy's formula} in regular lattices, which has been proved by Smirnov in the case of the triangular lattice (see \cite{smirnov_critical_2001}). We are interested in the \textit{universal} aspect of this scaling limit, in the sense that it is preserved for site, bond and face percolation on the UIHPT and the UIHPQ. Our main result is the following.

\begin{Th}\label{TheoremCP}
Let $a,b\geq 0$. We have for critical site, bond and face percolation models on the UIHPT and the UIHPQ,

$$\lim_{\lambda \rightarrow +\infty}{\mathbb{P}\left(\mathsf{C}^{*}(\lambda a, \lambda b)\right)}=\frac{1}{\pi}\arccos\left(\frac{b-a}{a+b}\right).$$ In other words, asymptotic crossing probabilities are equal for site, bond and face percolation on the UIHPT and the UIHPQ at their critical percolation threshold.
\end{Th}

The question of the universality of Cardy's formula for periodic plane graphs is an important open problem in probability theory, known as \textit{Cardy's universality conjecture}. Here, the randomness of the planar maps we consider makes the percolation models easier to study. 

\begin{Rks}
We believe that for both the computation of the site percolation threshold and the asymptotics of crossing probabilities, our methods apply in more general settings as long as the \textit{spatial Markov property} of Section \ref{SpatialMarkovProperty} holds. This includes the generalized half-planar maps of \cite{angel_classification_2015}, see also \cite{ray_geometry_2014} (in the ``tree-like" phase of these maps, the critical point is $1$ while in the ``hyperbolic" phase, although the crossing probability vanishes at criticality, its rate of convergence could be investigated), and also covers the models defined in \cite{bjornberg_recurrence_2014} and \cite{stephenson_local_2014}, which have recently been proved to satisfy a version of the spatial Markov property in \cite{budd_peeling_2015}. In particular, we are able to compute the site percolation threshold in the UIHPT, which is already provided in \cite{angel_percolations_2015}. The difference is that our method is less sensitive to boundary conditions, and also furnishes a universal formula for the site percolation thresholds, in the spirit of \cite{angel_percolations_2015}. We will discuss this in greater detail in Section \ref{SectionPCQ}.

Moreover, as we were finishing writing this paper, we became aware of the very recent preprint \cite{bjornberg_site_2014} by Björnberg and Stef\'ansson, that also deals with site percolation on the UIHPQ. This paper provides upper and lower bounds for the percolation threshold $p_{c,\rm{site}}^{\square}$, but not the exact value $5/9$. Our study of the universality of crossing probabilities is also totally independent of \cite{bjornberg_site_2014}.

\end{Rks}

\begin{Ack} I deeply thank Grégory Miermont for his help, and a lot of very stimulating discussions and ideas. I am also very grateful to Nicolas Curien for his remarks that led to Theorem \ref{ThUniversality}, and to the anonymous referee whose suggestions significantly improved this work.
\end{Ack}

\section{Random planar maps and percolation models}\label{SectionRPM}

We first recall the construction of the random planar maps we will focus on in the next part, and some important properties of these maps.

\subsection{Definitions and distributions on planar maps}

Let us first consider finite planar maps, i.e. proper embeddings of finite connected graphs in the sphere $\mathbb{S}^2$ (more precisely their equivalence class up to orientation-preserving homeomorphisms of the sphere). The faces of a planar map are the connected components of the complement of the embedding, and the degree of a face is the number of oriented edges this face is incident to (with the convention that the face incident to an oriented edge is the face on its left). Every planar map we consider is \textit{rooted}: there is a distinguished oriented edge called the \textit{root} of the map, and the tail vertex of this edge is the \textit{origin} of the map.

We focus more precisely on \textit{$p$-angulations} (the case $p=3$ corresponds to \textit{triangulations} and $p=4$ to \textit{quadrangulations}), i.e. finite planar maps whose faces all have the same degree $p$, and \textit{generalized} $p$-angulations, i.e. planar maps whose faces all have degree $p$, except possibly for a finite number of distinguished faces which can have arbitrary degrees. These faces are called \textit{external faces} of the $p$-angulation (by contrast with \textit{internal faces}), and are constrained to have a simple boundary (i.e. their embeddings are Jordan curves). In this setting, an inner vertex of the map is a vertex that do not belong to an external face. Finally, triangulations are supposed to be 2-connected (or type-2), that is to say, multiple edges are allowed but loops are not.

Recall that $*$ denote either of the symbols $\triangle$ or $\square$, where $\triangle$-angulations stands for type-2 triangulations and $\square$-angulations for quadrangulations. For $n\geq 1$, denote by $\mathcal{M}_n^*$ the set of generalized $*$-angulations that have $n$ inner vertices, so that $\mathcal{M}_f^*:=\cup_{n\geq 1}{\mathcal{M}_n^*}$ is the set of finite generalized $*$-angulations. By convention, we define $\mathcal{M}_0^*$ as the set of maps that have no inner vertices. We endow the set $\mathcal{M}_f^*$ of finite $*$-angulations with the \textit{local} topology, induced by the distance $d_{\text{loc}}$ defined for every $\boldsymbol{\mathrm{m}},\boldsymbol{\mathrm{m}}' \in \mathcal{M}_f^*$ by

$$d_{\mathrm{loc}}(\boldsymbol{\mathrm{m}},\boldsymbol{\mathrm{m}}'):=(1+\sup\{r\geq 0 : B_r(\boldsymbol{\mathrm{m}}) \sim B_r(\boldsymbol{\mathrm{m}}')\})^{-1},$$where $B_r(\boldsymbol{\mathrm{m}})$ is the planar map given by the ball of radius $r$ around the origin for the usual graph distance in the following sense: $B_0$ contains only the origin of the map, and $B_{r}$ is made of all the vertices at graph distance less than $r$ from the origin, with all the edges linking them. 

With this distance, $(\mathcal{M}_f^*,d_{\mathrm{loc}})$ is a metric space, and we define the set of (finite and infinite) $*$-angulations to be the completed space $\mathcal{M}^*$ of $(\mathcal{M}_f^*,d_{\mathrm{loc}})$. The elements of $\mathcal{M}_{\infty}^*:=\mathcal{M}^* \setminus \mathcal{M}_f^*$ can also be seen as infinite planar maps, in the sense that they can be defined as the proper embedding of an infinite, locally finite graph into a non-compact surface, dissecting the latter into a collection of simply connected regions (see the Appendix of \cite{curien_view_2013} for greater details). 

The first result of convergence in law of random maps for the local topology (Theorem 1.8 in \cite{angel_uniform_2003} $(\triangle)$, Theorem 1 in \cite{krikun_local_2005} $(\square)$) states as follows. For $n \geq 0$, let $\nu^*_n$ be the uniform measure on the set of $*$-angulations with $n$ vertices $(\triangle)$, respectively $n$ faces $(\square)$. We have

$$\nu_{n}^* \underset{n \rightarrow +\infty}{\Longrightarrow}  \nu_{\infty}^*,$$in the sense of weak convergence, for the local topology. The probability measure $\nu_{\infty}^*$ is supported on infinite $*$-angulations of the plane, and called the law of the Uniform Infinite Planar $*$-angulation (UIPT and UIPQ respectively in short). Note that there is an alternative construction in the quadrangular case, for which we refer to \cite{chassaing_local_2006}, \cite{menard_two_2010} and \cite{curien_view_2013}.

\vspace{3mm}

In the next part, we will focus on a slightly different model introduced by Angel in \cite{angel_scaling_2004}, that has an infinity boundary and nicer properties.

For $m\geq 2$, a (finite or infinite) generalized $*$-angulation and a unique external face of degree $m$ is called $*$-angulation of the $m$-gon. For $n\geq 0$ and $m\geq 2$, we denote by $\mathcal{M}_{n,m}^*$ the set of $*$-angulations of the $m$-gon with $n$ inner vertices \textbf{rooted on the boundary}, and $\varphi_{n,m}^*$ its cardinality. By convention, $\mathcal{M}_{0,2}^*=\mathcal{M}_{0}^*$. Note that $\varphi_{n,m}^{\square}=0$ for $m$ odd, so that we implicitly restrict ourselves to the cases where $m$ is even for quadrangulations. The quantity $\varphi_{n,m}^*$ being positive and finite for $n\geq 0$ and $m\geq 2$ (see \cite{tutte_enumeration_1968} for exact enumerative formulas), we can define the uniform probability measure on $\mathcal{M}_{n,m}^*$, denoted by $\nu_{n,m}^*$. Asymptotics for the numbers $(\varphi_{n,m}^*)_{n\geq 0,m\geq 2}$ are known and universal in the sense that

$$\varphi_{n,m}^* \underset{n \rightarrow +\infty}{\sim} C_*(m)\rho_*^nn^{-5/2} \quad \text{ and } \quad C_*(m) \underset{m \rightarrow +\infty}{\sim} K_*\alpha_*^m\sqrt{m},$$ where $\rho_\triangle=27/2$, $\alpha_\triangle=9$, $\rho_\square=12$, $\alpha_\square=\sqrt{54}$ and $K_*>0$ (see for instance the work of Gao, or more precisely \cite{krikun_explicit_2007} for 2-connected triangulations, and \cite{bouttier_distance_2009} for quadrangulations).

\vspace{3mm}

We now recall the construction of the uniform infinite half-planar maps. The Uniform Infinite Planar $*$-angulation of the Half-Plane, or UIHP-$*$, is a probability measure supported on infinite half-planar $*$-angulations defined by the following limits. 

\begin{Thbis}{(Theorem 2.1 in \cite{angel_scaling_2004})} For every $m\geq 2$, we have

$$\nu_{n,m}^* \underset{n \rightarrow +\infty}{\Longrightarrow}  \nu_{\infty,m}^*,$$in the sense of weak convergence, for the local topology. The probability measure $\nu_{\infty,m}^*$ is supported on infinite $*$-angulations of the $m$-gon and called the law of the UIPT and UIPQ of the $m$-gon, respectively.
Moreover, we have

$$\nu_{\infty,m}^* \underset{m \rightarrow +\infty}{\Longrightarrow} \nu_{\infty,\infty}^*,$$in the sense of weak convergence, for the local topology. The probability measure $\nu_{\infty,\infty}^*$ is called the law of the Uniform Infinite Half-Planar $*$-angulation (half-plane UIPT and UIPQ, or also UIHPT and UIHPQ respectively in short).
\end{Thbis}

This result is stated in \cite{angel_scaling_2004} for triangulations but easily extends to the quadrangular case, for which an alternative construction is also provided in Section 6.1 of \cite{curien_uniform_2015}. See also \cite{angel_classification_2015} for a characterization of these distributions.

We now give important properties of the previous measure, which justifies in particular the naming ``half-plane". The proof of such a result is based upon the inductive construction of the measure $\nu_{\infty,\infty}^*$ by a \textit{peeling process}, as described in \cite{angel_scaling_2004}. This is also illustrated in Figure \ref{UIHPTView}.

\begin{Propbis}
The probability measure $\nu_{\infty,\infty}^*$ is supported on infinite $*$-angulations of the half-plane with infinite boundary, and rooted on the boundary. 

Namely, these maps can be defined as the proper embedding of an infinite, locally finite connected graph in the upper half-plane $\mathbb{H}$ such that all the faces are finite and have degree 3 $(*=\triangle)$, respectively degree 4 $(*=\square)$. (In the case of triangulations, there are also no self-loops.) In particular, the restriction of this embedding to the infinite boundary of $\mathbb{H}$ is isomorphic to the graph of $\mathbb{Z}$ and the triangles (respectively quadrangles) exhaust the whole half-plane. 

Finally, the probability measure $\nu_{\infty,\infty}^*$ enjoys a re-rooting invariance property, in the sense that it is preserved under shifts of the root by one edge on the boundary.
\end{Propbis}

\begin{figure}[h]
\begin{center}
\includegraphics[scale=1.6]{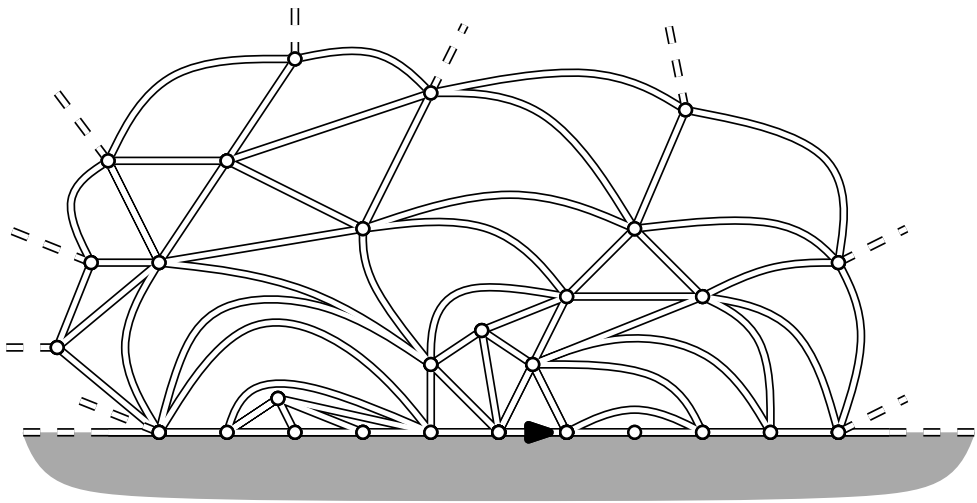}
\end{center}
\caption{An embedding of the UIHPT in the upper half-plane.}
\label{UIHPTView}
\end{figure}

In order to detail the outstanding properties of the measure $\nu_{\infty,\infty}^*$, we have to define another measure on finite planar maps, which gives an equal weight to maps with a fixed number of faces or, equivalently, of inner vertices. This measure is called the \textit{Boltzmann measure} (or \textit{free measure} in \cite{angel_growth_2003}). Let $m\geq 2$ and set

\begin{equation}\label{partfct}
  Z^*_m:=\sum_{n \geq 0}{\varphi_{n,m}^*\rho_*^{-n}},
\end{equation} which is exactly the generating function of $*$-angulations of the $m$-gon at the radius of convergence. Now, let us interpret the quantities $(Z^*_m)_{m\geq 2}$ as partition functions of a probability measure. This yields to the $*$-Boltzmann distribution of the $m$-gon, denoted by $\mu^*_m$, which is the probability measure on finite $*$-angulations of the $m$-gon defined for every $\boldsymbol{\mathrm{m}}\in \mathcal{M}_{n,m}^*$ by

\begin{equation}
\mu_m^*(\boldsymbol{\mathrm{m}}):=\frac{\rho_*^{-n}}{Z^*_m}. 
\end{equation} A random variable with law $\mu_m^*$ is called a Boltzmann $*$-angulation of the $m$-gon. Moreover, the asymptotic behaviour of $(Z^*_m)_{m\geq 2}$ is also known and given by \begin{equation}\label{equivZ}Z^*_m \underset{m \rightarrow +\infty}{\sim} \iota_*m^{-5/2}\alpha_*^m,\end{equation} for some constant $\iota_*>0$ and $\alpha_*$ as above.

The Boltzmann measures are particularly important because these objects satisfy a branching property, that we will identify on our uniform infinite $*$-angulations of the half-plane as the \textit{spatial} (or \textit{domain}) \textit{Markov property}. The reason why we will get this property on our infinite map is that the UIHP-$*$ can also be obtained as the limit of Boltzmann measures of the $m$-gon when $m$ becomes large.

\begin{Thbis}{(Theorem 2.1 in \cite{angel_scaling_2004})} We have $\mu^*_m \underset{m \rightarrow +\infty}{\Longrightarrow}  \nu_{\infty,\infty}^*$ in the sense of weak convergence, for the local topology.
\end{Thbis}

\subsection{Spatial Markov property and configurations}\label{SpatialMarkovProperty}

We now describe the so-called \textit{peeling} argument (introduced by Angel), whose principle is to suppose the whole map unknown and to reveal it face by face. 

Let us consider a random map $M$ which has the law of the UIHP-$*$, and a face $\mathsf{A}$ of $M$ which is incident to the root. Then, we \textit{reveal} or \textit{peel} the face $\mathsf{A}$, in the sense that we suppose the whole map unknown and work conditionally on the configuration of this face. We now consider the map $M\setminus\mathsf{A}$, which is the map $M$ deprived of the edges of $\mathsf{A}$ that belong to the boundary (in that sense, we also say that we \textit{peel} the root edge). By convention, the connected components of this map are the submaps obtained by cutting at a single vertex of $\mathsf{A}$ that belongs to the boundary.

We can now state a remarkable property of the UIHPT and the UIHPQ that will be very useful for our purpose, and which is illustrated in Figure \ref{SpatialMP}. This should be interpreted as a branching property of the measure $\nu^*_{\infty,\infty}$.

\begin{Thbis}{(Spatial Markov property, Theorem 2.2 in \cite{angel_scaling_2004})} Let $M$ be a random variable with law $\nu^*_{\infty,\infty}$, and $\mathsf{A}$ the face incident to the root edge of $M$. 

Then, $M\setminus\mathsf{A}$  has a unique infinite connected component, say $M'$, and at most one $(*=\triangle)$ or two $(*=\square)$ finite connected components, say $\tilde{M}_1$ and $\tilde{M}_2$. Moreover, $M'$ has the law $\nu^*_{\infty,\infty}$ of the UIHP-$*$, and $\tilde{M}_1$ and $\tilde{M}_2$ are Boltzmann $*$-angulations of a $m$-gon (for the appropriate value of $m\geq 2$ which is given by the configuration of the face $\mathsf{A}$). 

Finally, all those $*$-angulations are independent, in particular $M'$ is independent of $M\setminus M'$. This property still holds replacing the root edge by another oriented edge on the boundary, chosen independently of $M$.

\end{Thbis}

\begin{figure}[h]
\begin{center}
\includegraphics[scale=1.6]{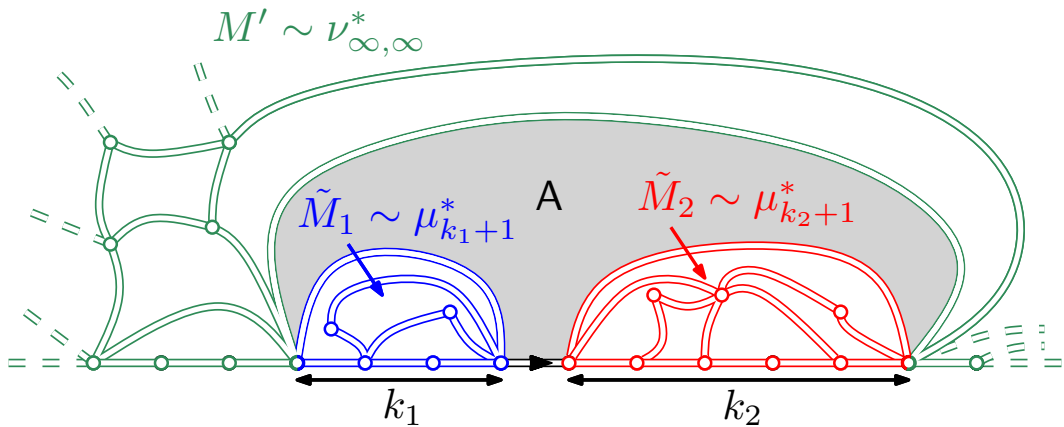}
\end{center}
\caption{The spatial Markov property.}
\label{SpatialMP}
\end{figure}

The peeling argument can now be extended to a \textit{peeling} or \textit{exploration} process, by revealing a new face in the unique infinite connected component of the map deprived of the last discovered face. The spatial Markov property ensures that the configurations for this face have the same probability at each step, and the re-rooting invariance of the UIHP-$*$ allows a complete freedom on the choice of the next edge to peel (as long as it does not depend on the unrevealed part of the map). This is the key idea in order to study percolation on the maps we consider, and has been used extensively in \cite{angel_scaling_2004} and \cite{angel_percolations_2015}.

\vspace{3mm}

We now describe all the possible configurations for the face $\mathsf{A}$ incident to the root in the UIHP-$*$, and the corresponding probabilities. Let us introduce some notations. On the one hand, some edges may lie on the boundary of the infinite connected component of $M\setminus\mathsf{A}$. These edges are called \textit{exposed} edges, and the (random) number of exposed edges is denoted by $\mathcal{E}^*$. On the other hand, some edges of the boundary may be enclosed in a finite connected component of the map $M\setminus\mathsf{A}$. We call them \textit{swallowed} edges, and the number of swallowed edges is denoted by $\mathcal{R}^*$. We will sometimes use the notation $\mathcal{R}_l^*$ (respectively $\mathcal{R}_r^*$) for the number of swallowed edges on the left (respectively right) of the root edge. (Note that the definitions of these random variables are robust with respect to the map model, but we may sometimes add the symbol $*$ to avoid confusion when the law of the map is not explicitly mentioned.) Finally, in the infinite maps we consider, an inner vertex is a vertex that do not belong to the boundary. The results are the followings. 

\begin{Prop}{(Triangulations case, \cite{angel_percolations_2015})}\label{tricase} There exists two configurations for the triangle incident to the root in the UIHPT (see Figure \ref{configT}).

\begin{enumerate}

\item The third vertex of the face is an inner vertex $(\mathcal{E}^{\triangle}=2,\mathcal{R}^{\triangle}=0)$. This event has probability $q^{\triangle}_{-1}=2/3$.

\item The third vertex of the face is on the boundary of the map, $k\geq 1$ edges on the left (respectively right) of the root $(\mathcal{E}^{\triangle}=1,\mathcal{R}^{\triangle}=k)$. This event has probability $q^{\triangle}_{k}=Z^{\triangle}_{k+1}9^{-k}$. 

\end{enumerate}

\begin{figure}[h!]
\begin{center}
\includegraphics[scale=1.6]{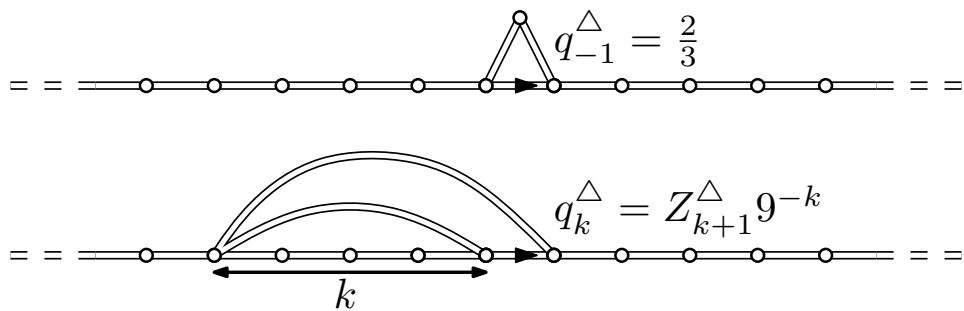}
\end{center}
\caption{The configurations of the triangle incident to the root (up to symmetries).}
\label{configT}
\end{figure}

\end{Prop}

\begin{Prop}{(Quadrangulations case, \cite{angel_percolations_2015})}\label{quadcase} There exists three configurations for the quadrangle $\mathsf{A}$ incident to the root in a map $M$ which has the law of the UIHPQ (see Figure \ref{configQ}).

\begin{enumerate}

\item The face has two inner vertices $(\mathcal{E}^{\square}=3,\mathcal{R}^{\square}=0)$. This event has probability $q^{\square}_{-1}=3/8$

\item The face has three vertices on the boundary of the map, the third one being $k\geq 0$ edges on the left (respectively right) of the root. This event has probability $q_k^{\square}$, given by the following subcases:

\begin{itemize}

\item If $k$ is odd, the fourth vertex belongs to the infinite connected component of $M\setminus\mathsf{A}$ $(\mathcal{E}^{\square}=2,\mathcal{R}^{\square}=k)$. Then

$$q_k^{\square}=\frac{Z^{\square}_{k+1}\alpha_{\square}^{1-k}}{\rho_{\square}}.$$

\item If $k$ is even, the fourth vertex belongs to the finite connected component of $M\setminus\mathsf{A}$ $(\mathcal{E}^{\square}=1,\mathcal{R}^{\square}=k)$. Then $$q_{k}^{\square}=\frac{Z^{\square}_{k+2}\alpha_{\square}^{-k}}{\rho_{\square}}.$$

\end{itemize}

\item The face has all of its four vertices on the boundary of the map, and the quadrangle defines two segments along the boundary of length $k_1$ and $k_2$ - both odd $(\mathcal{E}^{\square}=1,\mathcal{R}^{\square}=k_1+k_2)$. This event has probability $q_{k_1,k_2}^{\square}=Z^{\square}_{k_1+1}Z^{\square}_{k_2+1}\alpha_{\square}^{-k_1-k_2}$. 

(This configuration should be splitted in two subcases, depending on whether or not the vertices of the face are all on the same side of the root edge - up to symmetries).

\end{enumerate}

\begin{figure}[h]
\begin{center}
\includegraphics[scale=1.6]{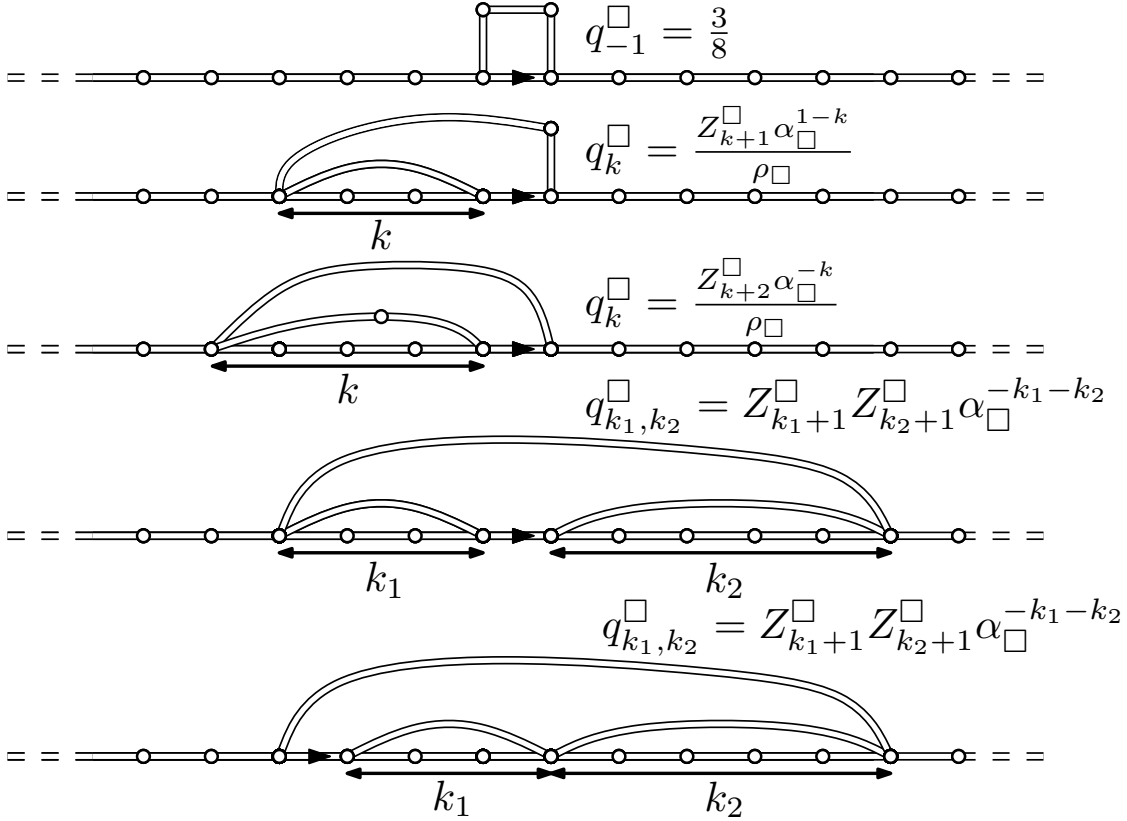}
\end{center}
\caption{The configurations of the quadrangle incident to the root (up to symmetries).}
\label{configQ}
\end{figure}

\end{Prop}

 We finally give the expectations of $\mathcal{E}^*$ and $\mathcal{R}^*$, obtained by direct computation using exact formulas for the partition function (\ref{partfct}), see Section 2.2 in \cite{angel_percolations_2015}.

\begin{Prop}{(Proposition 3 in \cite{angel_percolations_2015})}\label{PropACExp} We have
\begin{itemize}
\item $\mathbb{E}(\mathcal{E}^{\triangle})=5/3$ \textrm{ and } $\mathbb{E}(\mathcal{R}^{\triangle})=2/3$.
\item $\mathbb{E}(\mathcal{E}^{\square})=2$ \textrm{ and } $\mathbb{E}(\mathcal{R}^{\square})=1$.
\end{itemize}
Moreover, the distribution of $\mathcal{E}^{\square}$ can be explicitly computed.
\begin{equation*}
\mathcal{E}^{\square}=
\left\lbrace
\begin{array}{ccc}
3  & \mbox{with probability} & 3/8\\
2 & \mbox{with probability} & 1/4\\
1 & \mbox{with probability} & 3/8\\
\end{array}\right..
\end{equation*}

\end{Prop}

\begin{Rk}
The configurations being completely symmetric, we have $\mathcal{R}_l^*\overset{(d)}{=} \mathcal{R}_r^*$, so that $\mathbb{E}(\mathcal{R}^*_l)=\mathbb{E}(\mathcal{R}^*_r)=\mathbb{E}(\mathcal{R}^*)/2$.
\end{Rk}

\subsection{Percolation models}

We now specify the percolation models we focus on. Recall that we are interested in Bernoulli percolation on the random maps we previously introduced, i.e. every site (respectively edge, face) is open (we will say coloured black, or refer to the value 1 in the following) with probability $p$ and closed (coloured white, or taking value 0) otherwise, independently of every other sites (respectively edges, faces). Note that this colouring convention is the same as in \cite{angel_growth_2003}, but opposed to that of \cite{angel_percolations_2015}. The convention for face percolation is that two faces are adjacent if they share an edge.

Here is a more precise definition of the probability measure $\mathbb{P}_p$ induced by our model. Denote by $\mathcal{M}$ the set of planar maps, and for a given map $\boldsymbol{\mathrm{m}}\in \mathcal{M}$, define the following measure on the set $\{0,1\}^{e(\boldsymbol{\mathrm{m}})}$ of colourings of this map (where $e(\boldsymbol{\mathrm{m}})$ is the set of the ``elements" (vertices, edges or faces) of $\boldsymbol{\mathrm{m}}$):

$$\mathcal{P}_p^{e(\boldsymbol{\mathrm{m}})}:=(p\delta_1+(1-p)\delta_0)^{\otimes {e(\boldsymbol{\mathrm{m}})}}.$$ We then define $\mathbb{P}_p$ as the following measure on the set $\left\lbrace(\boldsymbol{\mathrm{m}},c) : \boldsymbol{\mathrm{m}}\in \mathcal{M}, c\in\{0,1\}^{e(\boldsymbol{\mathrm{m}})}\right\rbrace$ of coloured maps:

$$\mathbb{P}_p(\mathrm{d}\boldsymbol{\mathrm{m}}\mathrm{d}c):=\nu_{\infty,\infty}^{*}(\mathrm{d}\boldsymbol{\mathrm{m}})\mathcal{P}_p^{e(\boldsymbol{\mathrm{m}})}(\mathrm{d}c).$$ In other words, $\mathbb{P}_p$ is the measure on coloured planar maps such that the map has the law of the UIHP-$*$ and conditionally on this map, the colouring is a Bernoulli percolation with parameter $p$. We slightly abuse notation here, since we denote by $\mathbb{P}_p$ the probability measure induced by every map and percolation model considered in this paper, but there is little risk of confusion - if there is, we assign the notation $*$ to the random variables we consider. In what follows, we will often work conditionally on the colouring of the boundary of the map, which we call the boundary condition.

\vspace{3mm}

We finally define the percolation threshold (or critical point) in this model. Denote by $\mathcal{C}$ the open percolation cluster of the origin of the map we consider (respectively the root edge for face percolation) and recall that the percolation event is the event that $\mathcal{C}$ is infinite. The percolation probability is defined for $p\in[0,1]$ by

$$\Theta^*(p):=\mathbb{P}_p(\vert \mathcal{C} \vert=+\infty).$$ A standard coupling argument proves that the function $\Theta^*$ is non-decreasing, so that there exists a critical point $p_c^*$, called the percolation threshold, such that

\begin{equation*}
\left\lbrace
\begin{array}{ccc}
\Theta^*(p)=0 & \mbox{if} & p<p^*_c\\
\Theta^*(p)>0 & \mbox{if} & p>p^*_c\\
\end{array}\right..
\end{equation*} Thus, the percolation threshold $p^*_c$ can also be defined by the identity $p^*_c:=\inf\{p\in [0,1] : \Theta^*(p)>0\}=\sup\{p\in [0,1] : \Theta^*(p)=0\}$. Note that both $\Theta^*$ and $p^*_c$ depend on the law of the infinite planar map and on the percolation model we consider.

\vspace{3mm}

In our setting, we can use the peeling argument to derive a zero-one law under $\mathbb{P}_p$ for every $p \in [0,1]$, in the sense that every event invariant to finite changes in the coloured map has probability 0 or 1 under $\mathbb{P}_p$. The idea is to assign i.i.d. random variables to each step of peeling and to apply Kolmogorov's zero-one law (see Theorem 7.2 in \cite{angel_growth_2003} for a proof in the full-plane case). A fortiori, we have that the probability of such an event is 0 or 1 under $\mathcal{P}_p^{e(\boldsymbol{\mathrm{m}})}$ for $\nu_{\infty,\infty}^{*}(\mathrm{d}\boldsymbol{\mathrm{m}})$-almost every map. Following usual arguments of percolation theory (see Theorem 1.11 in \cite{grimmett_percolation_1999}), this yields that the event that there exists an infinite open cluster has probability 0 or 1 under $\mathbb{P}_p$ and then if we introduce the function defined for $p \in [0,1]$ by

$$\Psi^*(p):=\mathbb{P}_p\left(\exists \ x\in e'(\boldsymbol{\mathrm{m}}) : \vert \mathcal{C}(x) \vert=+\infty \right),$$ where $\mathcal{C}(x)$ is the open percolation cluster of the site (respectively edge for face percolation) $x$ of the map, we get that 

\begin{equation*}
\left\lbrace
\begin{array}{ccc}
\Psi^*(p)=0 & \mbox{if} & \Theta^*(p)=0\\
\Psi^*(p)=1 & \mbox{if} & \Theta^*(p)>0\\
\end{array}\right..
\end{equation*} Of course, we also get this result under $\mathcal{P}_p^{e(\boldsymbol{\mathrm{m}})}$ for $\nu_{\infty,\infty}^{*}(\mathrm{d}\boldsymbol{\mathrm{m}})$-almost every map.

\vspace{3mm}

The map on which we study percolation being itself random, one should take care of the distinction between \textit{annealed} and \textit{quenched} statements concerning percolation. The above percolation threshold $p^*_c$ is annealed, or averaged on all maps, since it is also the infimum of values of $p$ such that a positive measure of maps under $\nu_{\infty,\infty}^{*}$ satisfy that the percolation event has positive probability under $\mathcal{P}_p^{e(\boldsymbol{\mathrm{m}})}$. We should then denote this threshold by $p^{(*,\mathrm{A})}_c$.

The quenched percolation threshold $p^{(*,\mathrm{Q})}_c$ is defined as the infimum of values of $p$ such that for $\nu_{\infty,\infty}^{*}(\mathrm{d}\boldsymbol{\mathrm{m}})$\textit{-almost every map}, the percolation event has positive probability under $\mathcal{P}_p^{e(\boldsymbol{\mathrm{m}})}$. In other words, if we introduce the function defined for every $p \in [0,1]$ by

$$\tilde{\Theta}^*(p):=\nu_{\infty,\infty}^{*}\left(\left\lbrace \boldsymbol{\mathrm{m}} \in \mathcal{M} : \mathcal{P}_p^{e(\boldsymbol{\mathrm{m}})} \left( \vert \mathcal{C} \vert=+\infty \right)>0 \right\rbrace\right),$$ then the critical thresholds $p^{(*,\mathrm{A})}_c$ and $p^{(*,\mathrm{Q})}_c$ are defined by

\begin{equation*}
\left\lbrace
\begin{array}{ccc}
\tilde{\Theta}^*(p)=0 & \mbox{if} & p<p_c^{(*,\mathrm{A})}\\
\tilde{\Theta}^*(p)=1 & \mbox{if} & p>p_c^{(*,\mathrm{Q})}\\
\tilde{\Theta}^*(p)\in (0,1) & \mbox{otherwise} & \\
\end{array}\right..
\end{equation*} These can be proved to exist by using again the coupling argument. Now, if $p>p_c^{(*,\mathrm{A})}$, we have from the above observation that 

$$\Psi^*(p)=\int_{\mathcal{M}}{\mathcal{P}_p^{e(\boldsymbol{\mathrm{m}})}\left(\exists \ x\in e(\boldsymbol{\mathrm{m}}) : \vert \mathcal{C}(x) \vert=+\infty\right)\nu_{\infty,\infty}^{*}(\mathrm{d}\boldsymbol{\mathrm{m}})}=1,$$ which yields that $\mathcal{P}_p^{e(\boldsymbol{\mathrm{m}})}\left(\exists \ x\in e(\boldsymbol{\mathrm{m}}) : \vert \mathcal{C}(x) \vert=+\infty\right)=1$ and thus $\mathcal{P}_p^{e(\boldsymbol{\mathrm{m}})} \left( \vert \mathcal{C} \vert=+\infty \right)>0$ for $\nu_{\infty,\infty}^{*}(\mathrm{d}\boldsymbol{\mathrm{m}})$-almost every map. This proves that $\tilde{\Theta}^*(p)=1$, and the identity $$p^{(*,\mathrm{A})}_c=p^{(*,\mathrm{Q})}_c.$$ In the next part, we will thus use the notation $p^*_c$ and give only annealed statements concerning percolation thresholds. Note that in particular, the absence of percolation at criticality of Theorem \ref{TheoremPCQ} implies that for $\nu_{\infty,\infty}^{*}(\mathrm{d}\boldsymbol{\mathrm{m}})$-almost every map, there is no percolation at the critical point under $\mathcal{P}_p^{e(\boldsymbol{\mathrm{m}})}$.

\subsection{Lévy $3/2$-stable process}\label{SectionLevyStable}

Let us define the (spectrally negative) $3/2$-stable process and give some important properties that will be used later. All the results can be found in \cite{bertoin_levy_1998}.

\begin{Def}The Lévy spectrally negative $3/2$-stable process is the Lévy process $(\mathcal{S}_t)_{t\geq 0}$ whose Laplace transform is given by $\mathbb{E}(e^{\lambda \mathcal{S}_t})=e^{t\lambda^{3/2}}$, for every $\lambda \geq 0$. Its Lévy measure is supported on $\mathbb{R}_{-}$ and given by

$$\Pi(dx)=\frac{3}{4\sqrt{\pi}}\vert x \vert^{-5/2}dx \mathbf{1}_{\{x<0\}}.$$ In particular, this process has no positive jumps. Finally, the Lévy spectrally negative $3/2$-stable process has a scaling property with parameter $3/2$, i.e for every $\lambda>0$ the processes $(\mathcal{S}_t)_{t\geq 0}$ and $(\lambda^{-3/2}\mathcal{S}_{\lambda t})_{t\geq 0}$ have the same law.
\end{Def}

Note that this process is also called Airy-stable process (ASP in short). We will need the so-called \textit{positivity parameter} of the process $\mathcal{S}$, defined by

$$\rho:= \mathbb{P}(\mathcal{S}_1\geq 0)$$ Applying results of \cite{bertoin_levy_1998} in our setting, we get the following.

\begin{Lem}{(Chapter 8 of \cite{bertoin_levy_1998})} The positivity parameter of the Lévy $3/2$-stable process is given by $\rho=2/3$.
\end{Lem}

This process will be very useful for our purpose because it is the scaling limit of a large class of random walks whose steps are in the domain of attraction of a stable distribution with parameter $3/2$.

\begin{Prop}{(Section 4 in \cite{angel_percolations_2015}, Chapter 8 of \cite{bingham_regular_1989}, \cite{bertoin_random_2006})}\label{stablerw} Let $X$ being a \textbf{centered} real-valued random variable such that $P(X>t)=o(t^{-3/2})$ and $P(X<-t)=ct^{-3/2}+o(t^{-3/2})$, for a positive constant $c$. Then, if $S$ is a random walk with steps distributed as $X$, i.e. $S_0=0$ and for every $n\geq 1$, $S_n=\sum_{i=1}^n{X_i}$ where the random variables $(X_i)_{i\geq 1}$ are independent and have the same law as X, we have

$$\left(\frac{S_{\lfloor \lambda t\rfloor}}{\lambda^{2/3}}\right)_{t\geq 0} \underset{\lambda \rightarrow +\infty}{\overset{(d)}{\longrightarrow}} \kappa(\mathcal{S}_t)_{t \geq 0},$$ in the sense of convergence in law for Skorokhod's topology, where $\kappa$ is an explicit constant (proportional to $c$).

\end{Prop}

We end this section with an important property of the $3/2$-stable process, which concerns the distribution of the so-called \textit{overshoot} at the first entrance in $\mathbb{R}_-$. It is a consequence of the fact that the so-called \textit{ladder height process} of $-(\mathcal{S}_t)_{t\geq 0}$ is a stable subordinator with index $1/2$ (see \cite{bertoin_levy_1998} for details). We use the notation $P_a$ for the law of the process started at $a$.

\begin{Prop}{(Section 3.3 in \cite{angel_scaling_2004}, Chapter 3 of \cite{bertoin_levy_1998}, Example 7 in \cite{doney_overshoots_2006})}\label{stableovershoot} Let $\tau:= \inf\{t\geq 0 : \mathcal{S}_t \leq 0\}$ denote the first entrance time of the $3/2$-stable process $\mathcal{S}$ in $\mathbb{R}_-$. Then, the distribution of the overshoot $\vert \mathcal{S}_{\tau}\vert$ of $\mathcal{S}$ at the first entrance in $\mathbb{R}_-$ is given for every $a,b>0$ by

$$P_{a}(\vert \mathcal{S}_{\tau} \vert > b)=\frac{1}{\pi}\arccos\left(\frac{b-a}{a+b}\right).$$ Moreover, the joint distribution of the undershoot and the overshoot $(\mathcal{S}_{\tau-},\vert \mathcal{S}_{\tau}\vert)$ of $\mathcal{S}$ at the first entrance in $\mathbb{R}_-$ is absolutely continuous with respect to the Lebesgue measure on $(\mathbb{R}_+)^2$.\end{Prop}

\section{Site percolation threshold on the UIHPQ}\label{SectionPCQ}

Throughout this section, we consider Bernoulli site percolation on a random map which has the law of the UIHPQ, and focus on the computation of the site percolation threshold, denoted by $p_{c,\rm{site}}^{\square}$.  

First of all, let us state roughly why is computing the site percolation threshold different on triangulations and quadrangulations. The main idea of the proof in the triangular case in \cite{angel_scaling_2004} is to follow the percolation interface between the open origin and the closed neighbours on its right, say, using an exploration process. When dealing with quadrangulations, with positive probability the revealed quadrangle has the configuration of Figure \ref{ExplQ}. From there, one should try to follow the interface drawn on the figure, but the circled black vertex should not be ignored, because it is also possibly part of the open percolation cluster of the origin. Roughly speaking, keeping track of those revealed open vertices on the boundary makes the size of the revealed part of the cluster hard to study through the exploration.

\begin{figure}[h]
\begin{center}
\includegraphics[scale=1.6]{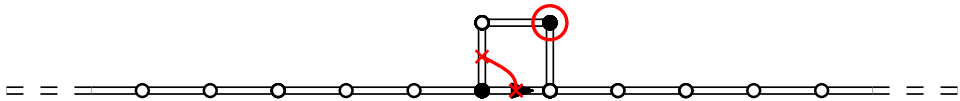}
\end{center}
\caption{Exploration process in the UIHPQ.}
\label{ExplQ}
\end{figure}

Our aim is now to prove Theorem \ref{TheoremPCQ}. In this statement, there is no condition on the initial colouring of the boundary, which is completely \textit{free} (a free vertex is by definition open with probability $p$, closed otherwise, independently of all other vertices in the map). In order to simplify the proof, but also for the purpose of Section \ref{SectionCrossingP}, we first work conditionally on the ``Free-Black-White" boundary condition presented in Figure \ref{initialcolouringPCQ}.

\begin{figure}[h]
\begin{center}
\includegraphics[scale=1.6]{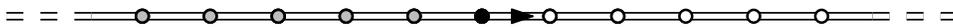}
\end{center}
\caption{The boundary condition for the site percolation threshold problem.}
\label{initialcolouringPCQ}
\end{figure}

The key here is to keep as much randomness as we can on the colour of the vertices and still to use an appropriate peeling process, following the ideas of \cite{angel_percolations_2015}. Since we work on the UIHPQ and in order to make the notation less cluttered, we omit the symbol ${\square}$ in what follows.

\subsection{Peeling process}

We want to reveal the map face by face in a proper way, which will define our peeling or exploration process. The strategy here is to reveal one by one the colour of the free vertices of the boundary, and to ``discard" or ``peel" the white vertices that are discovered in a sense we now make precise. To do so, we need an alternative peeling process, that we call ``vertex-peeling process". This process is well defined independently of the boundary condition, as long as we have a marked vertex on the boundary.

\begin{Alg}{(Vertex-peeling process)}\label{VertexPeelingProcess}
Consider a half-plane map which has the law of the UIHPQ and a marked vertex on the boundary.

Reveal the face incident to the edge of the boundary that is incident to the marked vertex on the left, and denote by $\mathsf{R}_r$ the number of swallowed edges on the right of this edge (without revealing the colour of the vertices that are discovered).

\begin{itemize}
\item If $\mathsf{R}_r>0$, the algorithm ends.
\item If $\mathsf{R}_r=0$, repeat the algorithm on the unique infinite connected component of the map deprived of the revealed face.
\end{itemize}

\end{Alg}

Let us now give the main properties of this algorithm, which is illustrated in Figure \ref{FigureVertexpeel} for the boundary condition that we will focus on in the next part (the arrow is pointed at the marked vertex).

\begin{Prop}\label{PropVertexPeel}The vertex-peeling process is well defined, in the sense that the marked vertex is on the boundary as long as the algorithm does not end, which occurs in finite time almost surely. Moreover, when the algorithm ends:

\begin{itemize}
\item The number of swallowed edges on the right of the marked vertex in the initial map has the law $\mathbb{P}_p(\mathcal{R}_r\in\cdot \mid \mathcal{R}_r>0)$ of the random variable $\mathcal{R}_r$ conditioned to be positive.

\item The unique infinite connected component of the map deprived of the revealed faces has the law of the UIHPQ and the marked vertex does not belong to this map. In that sense, this vertex has been peeled by the process.

\end{itemize}

\end{Prop}

\begin{proof} The conservation of the marked vertex, say $v$, on the boundary is a consequence of the fact that $\mathsf{R}_r = 0$ at each step as long as the process is not over. Moreover, the spatial Markov property implies that the sequence of swallowed edges to the right of $v$ is an i.i.d. sequence of random variables which have the same law as $\mathcal{R}_r$. The algorithm ends when the first positive variable in that sequence is reached (which happens in finite time almost surely), whose law is thus $\mathbb{P}_p(\mathcal{R}_r\in\cdot \mid \mathcal{R}_r>0)$. It also equals the number of swallowed edges to the right of $v$ in the initial map by construction. Finally, since $\mathsf{R}_r > 0$ at the last step, the marked vertex $v$ is not on the boundary of the unique infinite connected component we consider when the process ends. \end{proof}

\begin{figure}
\begin{center}
\includegraphics[scale=1.6]{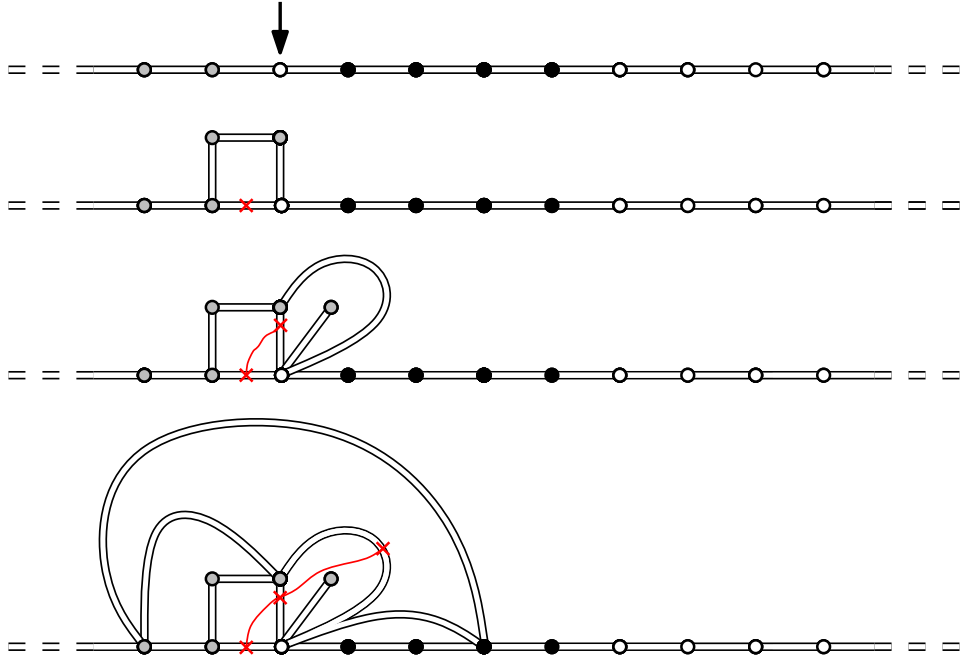}
\end{center}
\caption{An example of execution of the vertex-peeling process.}
\label{FigureVertexpeel}
\end{figure}

We are now able to describe the complete peeling process we focus on. In what follows, a ``Free-Black-White" boundary condition means in general that there are infinite free and white segments on the boundary (on the left and on the right respectively), and a finite black segment between them (Figure \ref{initialcolouringPCQ} is an example).

\begin{Alg}\label{PeelingPr}(Peeling process) Consider a half-plane map which has the law of the UIHPQ with a ``Free-Black-White" boundary condition.

Reveal the colour of the rightmost free vertex on the boundary.

\begin{itemize}
\item If it is black, repeat the algorithm.
\item If it is white, mark this vertex and execute the vertex-peeling process. 
\end{itemize} 
After each step, repeat the algorithm on the unique infinite connected component of the map deprived of the revealed faces. The algorithm ends when the initial open vertex of the boundary has been swallowed - i.e. does not belong to the map on which the algorithm should process.
\end{Alg}

\begin{Rk}
In the previous algorithms, the map we consider at each step of the peeling process is implicitly rooted at the next edge we have to peel, which is determined by the colouring of the boundary.\end{Rk}

As a consequence of the properties of the vertex-peeling process, we get that the whole peeling process is well defined, in the sense that the pattern of the boundary (Free-Black-White) is preserved as long as the algorithm does not end. Moreover, at each step, the planar map we consider has the law of the UIHPQ and does not depend on the revealed part of the map: the transitions of the peeling process are independent and have the same law. In particular, if we denote by $\mathcal{H}_n$, $c_n$ the number of swallowed edges at the right of the root edge and the colour of the revealed vertex at step $n$ of the exploration respectively, then $(\mathcal{H}_n,c_n)_{n\geq 0}$ are i.i.d. random variables.

\vspace{3mm}

The quantity we are interested in is the \textbf{length of the finite black segment} on the boundary of the map at step $n$ of the process, denoted by $B_n$. The process $(B_n)_{n \geq 0}$ is closely related to the percolation event by the following lemma.

\begin{Lem}\label{PercoEvent}
Denote by $\mathcal{C}$ the open cluster of the origin (open) vertex in our map. Then,

$$\left\lbrace B_n \underset{n \rightarrow + \infty}{\longrightarrow} + \infty \right\rbrace \subset \left\lbrace \vert \mathcal{C} \vert = + \infty \right\rbrace \quad \text{ and } \quad \left\lbrace \exists \ n \geq 0 : B_n = 0 \right\rbrace \subset \left\lbrace \vert \mathcal{C} \vert < + \infty \right\rbrace.$$

\end{Lem}

\begin{proof}
	Let us start with the first point. It is easy to see that as long as the peeling algorithm is not over, any black vertex discovered on the boundary is connected by an explicit open path to the origin of the map, and thus belongs to $\mathcal{C}$. Then, if $B_n$ goes to infinity (which implies that the algorithm do not end), we get that $\mathcal{C}$ is also infinite.
	
	The second statement is less obvious. The point is to see that at any step of the process, the open percolation cluster $\mathcal{C}$ of the origin in the initial map is infinite if and only if the open percolation cluster $\mathcal{C}_n$ of the finite black segment \textit{in the infinite connected component that we consider at step $n$ of the process} is itself infinite. To see that, note that when a white vertex is discovered and peeled, the black vertices of the boundary that are swallowed are enclosed in a finite Boltzmann map and thus provide a finite number of vertices to the percolation cluster $\mathcal{C}$ (out of $\mathcal{C}_n$). To conclude, note that if $B_n=0$ for a fixed $n\geq 0$, then at the last step the finite black segment is enclosed in a finite region of the map (see Figure \ref{FigureProofLemma}) and thus $\mathcal{C}_n$ is also finite, which ends the proof. \end{proof}
	
\begin{figure}[h]
\begin{center}
\includegraphics[scale=1.6]{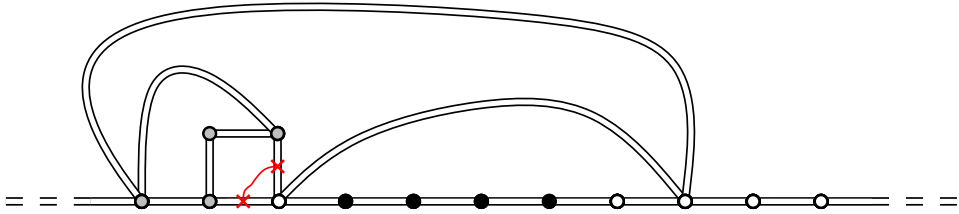}
\end{center}
\caption{The situation when the algorithm ends.}
\label{FigureProofLemma}
\end{figure}
	 
It is now sufficient to determine the behaviour of the process $(B_n)_{n \geq 0}$, which is remarkably simple as a consequence of the very definition of the peeling process and Proposition \ref{PropVertexPeel}.

\begin{Prop}\label{behaviourB}
The process $(B_n)_{n \geq 0}$ is a Markov chain with initial law $\delta_1$, whose transitions are given for $n\geq 0$ by
$$B_{n+1} = (B_n + 1 -\boldsymbol{1}_{c_n=0}\mathcal{H}_n)_+,$$ and which is absorbed at zero. Moreover, $(\mathcal{H}_n)_{n \geq 0}$ is a sequence of i.i.d. random variables and for every $n\geq 0$, conditionally on the event $\{c_n=0\}$, $\mathcal{H}_n$ has the law of $\mathcal{R}_r$ conditioned to be positive.
\end{Prop}

In particular, the process $(B_n)_{n \geq 0}$ has the same law as a random walk started at 1 and killed at its first entrance in $\mathbb{Z}_-$, with steps distributed as the random variable $1-\boldsymbol{1}_{c_0=0}\mathcal{H}_0$.

\subsection{Computation of the percolation threshold}

We now compute the percolation threshold for the ``Free-Black-White" boundary condition of Figure \ref{initialcolouringPCQ}.

\begin{Prop}\label{pcquad}
For Bernoulli site percolation on the UIHPQ and conditionally on the ``Free-Black-White" boundary condition with a single open vertex, we have

$$p_{c,\rm{site}}^{\square}=\frac{5}{9}.$$Moreover, there is no percolation at the critical point almost surely: $\Theta^{\square}_{\rm{site}}\left(p_{c,\rm{site}}^{\square}\right)=0$.
\end{Prop}

\begin{proof}
The quantity which rules the behaviour of $(B_n)_{n\geq 0}$ is the expectation of its steps, $\mathbb{E}_p(1 -\boldsymbol{1}_{c_0=0}\mathcal{H}_0)$. It is obvious that $\mathbb{E}_p(c_0)=p$. Now the law of $\mathcal{H}_0$ conditionally on $\{c_0=0\}$ do not depend on $p$ by construction and we have (see Proposition \ref{PropACExp} and the following remark for details):

$$\mathbb{E}_p(\mathcal{H}_0\mid c_0=0)=\mathbb{E}(\mathcal{R}_r\vert \mathcal{R}_r>0)=\frac{\mathbb{E}(\mathcal{R}_r\boldsymbol{1}_{\{\mathcal{R}_r>0\}})}{\mathbb{P}(\mathcal{R}_r>0)}=\frac{\mathbb{E}(\mathcal{R}_r)}{\mathbb{P}(\mathcal{R}_r>0)}=\frac{9}{4}.$$Thus, 

$$\mathbb{E}_p(1-\boldsymbol{1}_{c_0=0}\mathcal{H}_0)=1-(1-p)\mathbb{E}_p(\mathcal{H}_0\mid c_0=0)=1-\frac{9}{4}(1-p)$$ We get that $\mathbb{E}_p(1 -\boldsymbol{1}_{c_0=0}\mathcal{H}_0)=0$ if and only if $p=5/9$. In the case where $p\neq 5/9$, standard arguments on the behaviour of simple random walks combined with Lemma \ref{PercoEvent} yield the first statement. Finally, when $p=5/9$, the random walk with steps distributed as $1 -\boldsymbol{1}_{c_0=0}\mathcal{H}_0$ is null recurrent, so that almost surely, there exists $n\geq 0$ such that $B_n=0$. This concludes the proof of the second assertion.\end{proof}

\begin{Rk}A crucial quantity in the above computation of the percolation threshold is $\mathbb{P}(\mathcal{R}_r>0)$. Let us describe in greater detail how to get the exact value for this probability. Recall that we work in the UIHPQ, and let $\mathcal{V}$ be the number of vertices of the face incident to the root that lie on the boundary. In particular, we have $\mathcal{V}\in \{2,3,4\}$. We also let $\mathcal{V}_l$ and $\mathcal{V}_r$ be the number of such vertices (strictly) on the left and on the right of the root edge respectively. Combined with the quantities $\mathcal{E}$ and $\mathcal{R}$, this allows to distinguish all the possible configurations of the face. Using the fact that the sum of the quantities $(q_k^{\square})_{k\geq 0}$ is the same for even and odd values and Proposition \ref{PropACExp}, we get

$$\mathbb{P}(\mathcal{E}=2)=\mathbb{P}(\mathcal{E}=1,\mathcal{V}<4)=\frac{1}{4}.$$ Thus, $\mathbb{P}(\mathcal{E}=1,\mathcal{V}=4)=1/8$, and using the symmetries of the configurations, we obtain

$$\mathbb{P}(\mathcal{E}=1,\mathcal{V}_l=\mathcal{V}_r=1)=\mathbb{P}(\mathcal{E}=1,\mathcal{V}_r=2)=\frac{1}{24} \quad \text{ and } \quad \mathbb{P}(\mathcal{E}=2,\mathcal{V}_r=1)=\frac{1}{8}.$$ In the case where $\mathcal{E}=1$ and $\mathcal{V}<4$, one should take care of the fact that the number of swallowed edges can be zero on both sides with positive probability. We get
	
$$\mathbb{P}(\mathcal{E}=1,\mathcal{V}<4,\mathcal{V}_r=1)=\frac{1}{8}-q^{\square}_0.$$ Now, 
	
$$\mathbb{P}(\mathcal{R}_r>0)=\mathbb{P}(\mathcal{V}_r>0)=\frac{1}{4}+\frac{1}{12}-q^{\square}_0=\frac{2}{9},$$ where we use the exact formulas for $q^{\square}_0$ and for the partition function (\ref{partfct}) given in Section 2.2 of \cite{angel_percolations_2015}, or in \cite{bouttier_distance_2009} (we find $Z^{\square}_2=4/3$, and thus $q^{\square}_0=1/9$).
\end{Rk}

\subsection{Universality of the percolation threshold}

We now want to discuss the \textit{universal} aspects of the above result. The first one is the universality of the percolation threshold with respect to the boundary condition. Namely, we now consider Bernoulli site percolation on the UIHPQ in the most natural setting which is the free boundary condition, and prove Theorem \ref{TheoremPCQ}.

\begin{proof}[Proof of Theorem \ref{TheoremPCQ}] First of all, we can work without loss of generality conditionally on the fact that the origin vertex of the map is open. We then use a peeling process which is heavily based on the previous one. 

Namely, we execute Algorithm \ref{PeelingPr} with the convention that we reveal the colour of the rightmost free vertex on the left of the origin (on the boundary).

The algorithm is paused when the finite black segment on the boundary has been completely swallowed. When this happens, two situations may occur, depending on the colour of the rightmost vertex $v_0$ of the boundary that is part of the last revealed face. By properties of the vertex-peeling process, such a vertex exists and lies on the right of the root edge (see Figure \ref{SituationTh1}).

\begin{enumerate}
\item If $v_0$ is white, then the open cluster of the origin is enclosed in a finite region of the map and the algorithm ends: the percolation event does not occur.
\item If $v_0$ is black, then the percolation event occurs if and only if the open cluster of $v_0$ in the unrevealed infinite connected component is itself infinite. Then, the algorithm goes on executing Algorithm \ref{PeelingPr} on this map, with $v_0$ as origin vertex. 
\end{enumerate} We let $(v_k)_{k\geq 0}$ be the sequence of vertices defined as $v_0$ at each time the algorithm is paused.

\begin{figure}[h]
\begin{center}
\includegraphics[scale=1.6]{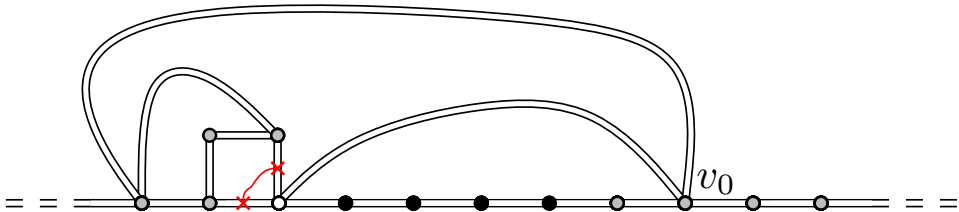}
\end{center}
\caption{The situation when the finite black segment on the boundary is swallowed.}
\label{SituationTh1}
\end{figure}

We are now able to conclude the proof. First, if $p> 5/9$, we know that $\Theta^{\square}_{\rm{site}}(p)>0$ using directly Proposition \ref{pcquad} and a standard monotone coupling argument. On the other hand, if $p\leq 5/9$, we know from Proposition \ref{pcquad} that the above algorithm is paused in finite time almost surely. By induction, we already stated that if there exists a vertex $v_k$ which is white, the percolation event cannot occur. However, the colouring of the vertices $(v_k)_{k\geq 0}$ form an i.i.d. sequence with Bernoulli law of parameter $1-p>0$, which yields that $\Theta^{\square}_{\rm{site}}(p)=0$ and the expected result.\end{proof}

Let us finally detail the universality of our methods with respect to the law of the map. First, the previous arguments can be easily adapted to the case of the UIHPT. In particular, we have that for a free boundary condition

$$p_{c,\rm{site}}^{\triangle}=\frac{1}{2},$$and that there is no percolation at the critical point almost surely. This result has already been proved in \cite{angel_percolations_2015} for a closed boundary condition (except the origin vertex). The interesting fact is that we can derive from our proof a universal formula for the site percolation thresholds, exactly as it was done in \cite{angel_percolations_2015}, Theorem 1, for bond and face percolation. Following this idea, we introduce the following notations:

$$\delta^*:=\mathbb{E}(\mathcal{R^*})=2\mathbb{E}(\mathcal{R}^*_r) \qquad \text{ and } \qquad \eta^*:=\mathbb{P}(\mathcal{R}^*_r>0),$$ where $*$ still denotes either of the symbols $\triangle$ or $\square$. From the results of Section \ref{SectionRPM} and the Remark following Proposition \ref{pcquad} we get that 

$$\eta^{\triangle}=\frac{1}{6}, \quad \eta^{\square}=\frac{2}{9}, \quad \delta^{\triangle}=\frac{2}{3} \quad \text{ and } \quad \delta^{\square}=1.$$ Note that the computation of $\eta^{\triangle}$ is immediate from the symmetry of the configurations and the fact that there are no self-loops. The arguments of the proof of Theorem \ref{TheoremPCQ} yield the following result.

\begin{Th}\label{ThUniversality}For site percolation on the UIHPT and the UIHPQ, we have that

$$p_{c,\rm{site}}^{*}=1-\frac{2\eta^*}{\delta^*}.$$

\end{Th} Of course, this formula is believed to hold in greater generality, for any infinite random half-planar map satisfying the spatial Markov property. Note that two parameters are needed to describe the site percolation threshold, while $\delta^*$ is sufficient for bond and face percolation as proved in \cite{angel_percolations_2015}.
  
\section{Scaling limits of crossing probabilities in half-plane random maps}\label{SectionCrossingP}

Throughout this section, we focus on the problem of scaling limits of crossing probabilities, and aim at generalizing the results of \cite{angel_scaling_2004}. Despite the fact that we also use a peeling process, the problem is much harder because the models we consider are less well-behaved than site percolation on the UIHPT.

More precisely, we consider site, bond and face percolation on the UIHPT and the UIHPQ, and work conditionally on the boundary condition represented on Figure \ref{colouring} (for the bond percolation case). In other words, the boundary is ``White-Black-White-Black", with two infinite segments and two finite ones, of lengths $\lfloor\lambda a\rfloor$ and $\lfloor\lambda b\rfloor$ respectively. The crossing event we focus on is the following: ``there exists a black path linking the two black segments of the boundary", or equivalently ``the two black segments are part of the same percolation cluster". We denote by $\mathsf{C}^*_{\rm{bond}}(\lambda a, \lambda b)$ (respectively $\mathsf{C}^*_{\rm{face}}(\lambda a, \lambda b)$ and $\mathsf{C}^*_{\rm{site}}(\lambda a, \lambda b)$) this event where $*$ denote either of the symbols $\triangle$ or $\square$ (see Figure \ref{colouring}).

\begin{figure}[h]
\begin{center}
\includegraphics[scale=1.6]{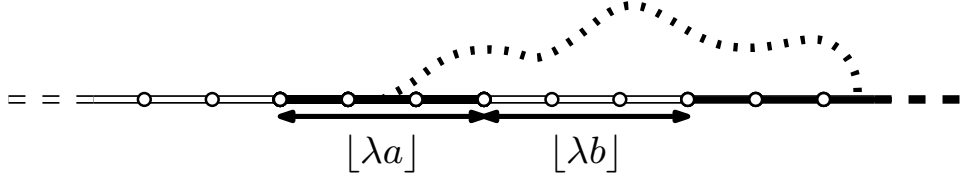}
\end{center}
\caption{The boundary condition and the crossing event $\mathsf{C}^*(\lambda a, \lambda b)$.}
\label{colouring}
\end{figure}

Again, although the dependence on the map and percolation models are rather supported by the probability measure in this setting, we keep using the notation $*$ for the sake of simplicity. The quantity we are interested in is the scaling limit of the crossing probability $\mathbb{P}_p(\mathsf{C}^*(\lambda a, \lambda b))$,

$$\lim_{\lambda \rightarrow +\infty}{\mathbb{P}_p\left(\mathsf{C}^*(\lambda a, \lambda b)\right)},$$ which we expect to be universal following Cardy's universality conjecture if we consider the percolation models at their critical point, i.e. when $p$ is exactly the percolation threshold. We recall those values that can be found in \cite{angel_percolations_2015} and the previous section.

\begin{Thbis}{(Theorem 1 in \cite{angel_percolations_2015}, \cite{angel_growth_2003}, Theorem \ref{TheoremPCQ})} The percolation thresholds on the UIHPT are given by

$$p_{c,\rm{site}}^{\triangle}=\frac{1}{2}, \quad p_{c,\rm{bond}}^{\triangle}=\frac{1}{4} \quad \text{and} \quad p_{c,\rm{face}}^{\triangle}=\frac{4}{5}.$$The percolation thresholds on the UIHPQ are given by

$$p_{c,\rm{site}}^{\square}=\frac{5}{9},\quad p_{c,\rm{bond}}^{\square}=\frac{1}{3} \quad \text{and} \quad p_{c,\rm{face}}^{\square}=\frac{3}{4}.$$

\end{Thbis}

Starting from now, the probability $p$ that an edge (respectively face, vertex) is open is always set at $p_{c,\rm{bond}}^*$ (respectively $p_{c,\rm{face}}^*$, $p_{c,\rm{site}}^*$), in every model we consider. Most of the arguments being valid for both the UIHPT and the UIHPQ, we will treat those cases simultaneously and omit the notation $*$, underlying the differences where required. For technical reasons, we have to treat the cases of bond, face and site percolation separately, even though the proof always uses a similar general strategy. Our aim is now to prove Theorem 2.

\subsection{Crossing probabilities for bond percolation}\label{SectionBondCase}

\subsubsection{Peeling process and scaling limit}\label{PeelingScalingLimit}

We start with the case of bond percolation. In order to compute the crossing probability we are interested in, we again use a peeling process that we now describe. The purpose of this process is to follow a closed path in the dual map, which is the map that has a vertex in each face of the initial (primal) map, and edges between neighboring faces (see \cite{grimmett_percolation_1999}, Section 11.2 for details). The reason is that dual closed and primal open crossings between the segments of the boundary are almost surely complementary events.

The algorithm starts with revealing the face incident to the rightmost white edge of the infinite white segment (see Figure \ref{FirstEdge}), without discovering the colour of its incident edges, which are thus free. (The rightmost white edge is then peeled in the sense that it is no longer in the infinite connected component of the map deprived of the revealed face.)

\begin{figure}[h]
\begin{center}
\includegraphics[scale=1.6]{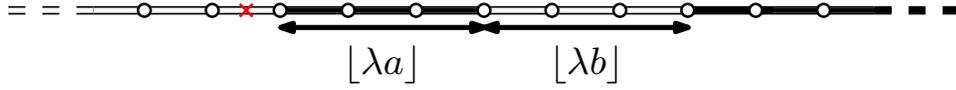}
\end{center}
\caption{The initial boundary and the first edge to peel.}
\label{FirstEdge}
\end{figure}

In general, the algorithm is well defined as soon as the boundary has a ``White-(Free)-Any" colouring, meaning that there is an infinite white segment and (possibly) a finite free segment on the left of the boundary, but no specific condition on the colouring of the right part. Although the right part will turn out to be coloured ``Black-White-Black" in what follows, it will be helpful to have an algorithm which is defined independently of the colouring of the right part of the boundary.

\begin{Alg}\label{AlgoBond}Consider a half-plane map which has the law of the UIHP-$*$ with a ``White-(Free)-Any" boundary condition.

\begin{itemize}

\item Reveal the colour of the rightmost free edge, if any:

\begin{itemize}
\item If it is black, repeat the algorithm.

\item If it is white, reveal the face incident to this edge (without revealing the colour of the other edges of the face).

\end{itemize} 

\item If there is no free edge on the boundary of the map, as in the first step, reveal the face incident to the rightmost white edge of the infinite white segment (without revealing the colour of the other edges of the face).
\end{itemize} After each step, repeat the algorithm on the UIHP-$*$ given by the unique infinite connected component of the current map, deprived of the revealed face. 
\end{Alg}

We will call \textit{revealed part} of the map all the vertices and edges that are not in this infinite connected component (even if all of them may not have been revealed, but rather swallowed in an independent finite connected component). The map we obtain is still implicitly rooted at the next edge we have to peel.

Let us now give some properties of Algorithm \ref{AlgoBond}, which are essentially the same as that of Algorithm \ref{PeelingPr}, apart from the fact that the algorithm never ends. We omit the proof of these properties, that are direct consequences of the spatial Markov property.

\begin{Prop}\label{PeelingProcess}
The peeling process is well defined, in the sense that the pattern of the boundary (White-(Free)-Any) is preserved. Moreover, at each step, the planar map we consider has the law of the UIHP-$*$ and is independent of the revealed part of the map. 
\end{Prop}

Starting from now, we denote by $\mathcal{E}_n$, respectively $\mathcal{R}_n,c_n,$ the number of exposed edges, respectively the number of swallowed edges and the color of the revealed edge at step $n$ of the exploration, for every $n\geq 0$. By convention, we set $c_n= 0$ when there is no free edge on the boundary and $\mathcal{E}_n=\mathcal{R}_n=0$ when no edge is peeled $(c_n=1)$. Recall that $\mathcal{R}_{l,n}$ (respectively $\mathcal{R}_{r,n}$) denotes the number swallowed edges at the left (respectively right) of the root edge at step $n$ of the exploration. The quantity we are interested in is the size of the finite black segment at step $n$. In order to study this quantity, we now introduce two random processes.

\vspace{3mm}

First, for every $n\geq 0$, we let $F_n$ be the \textbf{length of the free segment} on the boundary at step $n$ of the peeling process. Then, we have $F_0=0$ and for every $n\geq 0$,

\begin{equation}\label{DefF}
F_{n+1}=\left\lbrace
\begin{array}{ccc}
F_n-1 & \mbox{if} & c_n=1\\
\mathcal{E}_n+(F_n-\mathcal{R}_{l,n}-1)_+ & \mbox{if} & c_n=0\\
\end{array}\right..
\end{equation} The process $(F_n)_{n\geq 0}$ is a Markov chain with respect to the canonical filtration of the exploration process, and we have $c_n= 0$ when $F_n=0$.

\begin{Rk} The Markov chain $(F_n)_{n\geq 0}$ takes value zero at step $n$ only if we have $F_{n-1}=1$ and $c_n=1$. Moreover, if the whole free segment is swallowed by a jump of the peeling process at step $n$, then $F_n=\mathcal{E}_n$ (see Propositions \ref{tricase} and \ref{quadcase}).
\end{Rk}
 
Then, we let $B_0=\lfloor \lambda a \rfloor$ and for every $n\geq 0$,

\begin{equation}\label{DefB}
B_{n+1}=\left\lbrace
\begin{array}{ccc}
B_n+1 & \mbox{if} & c_n=1\\
B_n-\mathcal{R}_{r,n} & \mbox{if} & c_n=0\\
\end{array}\right..
\end{equation} Like $(F_n)_{n\geq 0}$, the process $(B_n)_{n\geq 0}$ is a Markov chain with respect to the canonical filtration of the exploration process. Moreover, if we denote by $T:=\inf\{n \geq 0 : B_n \leq 0 \}$ the first entrance time of $(B_n)_{n\geq 0}$ in $\mathbb{Z}_-$, then for every $n\in \{0,1,\ldots,T-1\}$, $B_n$ is the \textbf{length of the finite black segment} at step $n$ of the exploration process. 

An important point is that for every $n<T$, the edges of the finite black segment are all part of the open cluster of the initial segment, while for $n \geq T$, the initial finite black segment has been swallowed by the peeling process. (However, the algorithm can continue because the infinite white and finite free segments are still well defined.) 

As we will see, and as was stressed in \cite{angel_scaling_2004}, the so-called \textit{overshoot} $B_{T}$ of the process $(B_n)_{n\geq 0}$ at the first entrance in $\mathbb{Z}_-$ is the central quantity, which governs the behaviour of crossing events.

We now need to describe the law of the random variables involved in the previous definitions of the processes $(F_n)_{n\geq 0}$ and $(B_n)_{n\geq 0}$. This can be done using the very definition of the peeling algorithm and Section \ref{SpatialMarkovProperty}. Note that one has to be careful with the fact that $\mathcal{E}_n=\mathcal{R}_n=0$ when $c_n=1$. Thus, we let $(t_k)_{k\geq 1}$ be the sequence of consecutive stopping times such that $c_{t_k}=0$ (i.e. the times at which a face is peeled), and $(s_k)_{k\geq 1}$ the sequence of consecutive stopping times such that $F_{s_k}>0$ (i.e. the times at which there is an edge whose colour is revealed). In what follows, we let $(\mathcal{E},\mathcal{R}_l,\mathcal{R}_r)$ be defined as in Section \ref{SectionRPM} (notation $*$ omitted). A simple application of the spatial Markov property yields the following result.

\begin{Lem}\label{LawER}The random variables $(\mathcal{E}_{t_k},\mathcal{R}_{l,t_k},\mathcal{R}_{r,t_k})_{k\geq 1}$ are i.i.d. and have the same law as $(\mathcal{E},\mathcal{R}_l,\mathcal{R}_r)$. Moreover, the random variables $(c_{s_k})_{k\geq 1}$ are i.i.d. and have the Bernoulli law of parameter $p^*_{c,\text{bond}}$.
\end{Lem}

In the following, we denote by $c$ a random variable which has the Bernoulli law of parameter $p^*_{c,\text{bond}}$, independent of $(\mathcal{E},\mathcal{R}_l,\mathcal{R}_r)$. We are now able to describe the behaviour of the processes $(F_n)_{n\geq 0}$ and $(B_n)_{n\geq 0}$.

\vspace{3mm}

\textbf{The process $(F_n)_{n\geq 0}$.} The Markov chain $(F_n)_{n\geq 0}$ is not exactly a random walk, but has the same behaviour when it is far away from zero. More precisely, note that $F_1>0$ almost surely and let $\hat{\sigma}:=\inf\{n \geq 1 : F_n-1-\mathbf{1}_{\{c_n=0\}}\mathcal{R}_{l,n}< 0 \}$ be the first time at which the free segment is swallowed (possibly by a revealed face). Note that $F_{\hat{\sigma}}$ is not zero in general, but if $F_n=0$ for $n\geq 1$, then $\hat{\sigma}\leq n$. By construction, as long as $n\leq\hat{\sigma}$, we have

$$F_n=F_1+\sum_{k=1}^{n-1}{\left(\mathbf{1}_{\{c_k=0\}}(\mathcal{E}_k-\mathcal{R}_{l,k})-1\right)}.$$ 

Thus, $(F_n-F_1)_{1\leq n \leq \hat{\sigma}}$ is a random walk killed at the random time $\hat{\sigma}$, whose steps are distributed as $\hat{X}$ defined by

$$\hat{X}:=\mathbf{1}_{\{c=0\}}(\mathcal{E}-\mathcal{R}_l)-1,$$ where we use the definition of $\hat{\sigma}$ and Lemma \ref{LawER}. Moreover, we obtain from Proposition \ref{PropACExp} that

\begin{align*}
\mathbb{E}(\hat{X})&=-p^*_{c,\text{bond}}+(1-p^*_{c,\text{bond}})(\mathbb{E}(\mathcal{E}-\mathcal{R}_l)-1)=
\left\lbrace
\begin{array}{ccc}
\frac{1}{3}-\frac{4}{3}p^*_{c,\text{bond}} \ (*=\triangle) \\
\\
\frac{1}{2}-\frac{3}{2}p^*_{c,\text{bond}} \ (*=\square) \\
\end{array}\right..
\end{align*}
Using the percolation thresholds given at the beginning of Section \ref{SectionCrossingP}, we get that $\mathbb{E}(\hat{X})=0$. As a consequence, since $\hat{\sigma}<\sigma:=\inf\{k\geq 1 : F_k = 0\}$, we get that $\hat{\sigma}$ is almost surely finite. It is finally easy to check that the random variable $\hat{X}$ satisfies the conditions of Proposition \ref{stablerw}. In particular, if we denote by $(\hat{S}_n)_{n \geq 0}$ a random walk with steps distributed as $\hat{X}$ we get

\begin{equation}\label{CvgSkoShat}
\left(\frac{\hat{S}_{\lfloor \lambda t\rfloor}}{\lambda^{2/3}}\right)_{t \geq 0} \underset{\lambda \rightarrow +\infty}{\overset{(d)}{\longrightarrow}} \kappa(\mathcal{S}_t)_{t \geq 0},
\end{equation} in the sense of convergence in law for Skorokhod's topology, where $\mathcal{S}$ is the Lévy $3/2$-stable process of Section \ref{SectionLevyStable}, and $\kappa>0$.

\vspace{3mm}

\textbf{The process $(B_n)_{n\geq 0}$.} Exactly as before, $(B_n)_{n\geq 0}$ is not a random walk, but behaves the same way when $F_n$ is not equal to zero. Let us be more precise here. First, let $\sigma_0=0$, and inductively for every $k\geq 1$,

\begin{equation}\label{defsigma}
  \sigma_k:=\inf\{n\geq \sigma_{k-1}+1 : F_n=0 \}.
\end{equation} From the above description of the process $(F_n)_{n\geq 0}$, we know that started from any initial size, the free segment is swallowed in time $\hat{\sigma}$, which is almost surely finite. When this happens, either a black edge has been revealed and $F_{\hat{\sigma}}=0$, or a revealed face has swallowed the free segment. In this case, with probability bounded away from zero, $F_{\hat{\sigma}+1}=0$. Applying the strong Markov property, we get that all the stopping times $(\sigma_k)_{k\geq 0}$ are almost surely finite.

Since $F_1>0$ by construction, as long as $1\leq n\leq \sigma=\sigma_1$ we have

$$B_n=B_1+\sum_{k=1}^{n-1}{\left(\mathbf{1}_{\{c_k=1\}}-\mathbf{1}_{\{c_k=0\}}\mathcal{R}_{r,k}\right)}.$$ 

Thus, $(B_n-B_1)_{1\leq n \leq \sigma}$ is a random walk killed at the random time $\sigma$, whose steps are distributed as $X$ defined by  

$$X:=\mathbf{1}_{\{c=1\}}-\mathbf{1}_{\{c=0\}}\mathcal{R}_{r}.$$ From the strong Markov property, this also holds for the processes $(B_{\sigma_k+n}-B_{\sigma_k+1})_{1\leq n \leq \sigma_{k+1}-\sigma_k}$, noticing that $F_{\sigma_k+1}>0$ for every $k\geq 0$ by definition. As a consequence of Proposition \ref{PropACExp}, we get that

\begin{align*}
\mathbb{E}(X)&=p^*_{c,\text{bond}}-(1-p^*_{c,\text{bond}})\mathbb{E}(\mathcal{R}_r)
=\left\lbrace
\begin{array}{ccc}
\frac{4}{3}p^{\triangle}_{c,\text{bond}}-\frac{1}{3}  \ (*=\triangle) \\
\\
\frac{3}{2}p^{\square}_{c,\text{bond}}-\frac{1}{2}  \  (*=\square) \\
\end{array}\right..
\end{align*}
Since we work at the critical point, $\mathbb{E}(X)=0$ for both the UIHPT and the UIHPQ, and properties of the random variable $X$ also yield that if $(S_n)_{n \geq 0}$ is a random walk with steps distributed as $X$,

\begin{equation}\label{CvgSkoS}
\left(\frac{S_{\lfloor \lambda t\rfloor}}{\lambda^{2/3}}\right)_{t \geq 0} \underset{\lambda \rightarrow +\infty}{\overset{(d)}{\longrightarrow}} \kappa(\mathcal{S}_t)_{t \geq 0},
\end{equation} in the sense of convergence in law for Skorokhod's topology, where $\mathcal{S}$ is the Lévy $3/2$-stable process of Section \ref{SectionLevyStable} and $\kappa>0$.

\vspace{5mm}

We are now interested in the scaling limit of the process $(B_n)_{n\geq 0}$. We want to prove the following result:

\begin{Prop}\label{cvgscallim} We have, in the sense of convergence in law for Skorokhod's topology: 

$$\left(\frac{B_{\lfloor \lambda^{3/2} t\rfloor}}{\lambda}\right)_{t \geq 0} \underset{\lambda \rightarrow +\infty}{\overset{(d)}{\longrightarrow}}\kappa(\mathcal{S}_t)_{t \geq 0},$$ where $\mathcal{S}$ is the Lévy spectrally negative $3/2$-stable process started at $a$, and $\kappa>0$.
\end{Prop}

The idea is to couple a random walk $(S_n)_{n\geq 0}$ with steps distributed as the random variable $X$ and the process $(B_n)_{n\geq 0}$ in such a way that $S_n$ and $B_n$ are close. To do so, let us set $S_0=B_0$, and introduce an independent sequence $(\beta_n)_{n \geq 0}$ of i.i.d. random variables with Bernoulli law of parameter $p^*_{c,\text{bond}}$. We then define $(S_n)_{n\geq 0}$ recursively as follows. For every $n\geq 0$:

\begin{align*}
S_{n+1}
=S_n+\left\lbrace
\begin{array}{ccc}
B_{n+1}-B_n=\mathbf{1}_{\{c_n=1\}}-\mathbf{1}_{\{c_n=0\}}\mathcal{R}_{r,n} &\mbox{ if }  F_n>0 \\
\\
\beta_n+(1-\beta_n)(B_{n+1}-B_n)=\mathbf{1}_{\{\beta_n=1\}}-\mathbf{1}_{\{\beta_n=0\}}\mathcal{R}_{r,n}  &\mbox{ if }  F_n=0 \\
\end{array}\right..
\end{align*} Otherwise said, as long as $F_n>0$, $(S_n)_{n\geq 0}$ performs the same steps as $(B_n)_{n\geq 0}$ and when $F_n=0$, $S_{n+1}$ takes value $S_n+1$ if $\beta_n=1$, and does the same transition as $(B_n)_{n\geq 0}$ otherwise. Using the previous discussion and the fact that $c_n=0$ when $F_n=0$, we get that $(S_n)_{n\geq 0}$ is a random walk started from $B_0$ with steps distributed as the random variable $X$, as wanted.

\vspace{3mm}

Let us also define the set $\Xi_n:=\{0\leq k < n : F_n=0\}$ of zeros of $(F_n)_{n\geq 0}$ before time $n\geq 0$, and set $\xi_n:=\#\Xi_n$ for every $n\geq 0$. The steps of $(S_n)_{n\geq 0}$ and $(B_n)_{n\geq 0}$ are the same except when $F_n=0$, in which case, since $S_{n+1}-S_n$ for every $n\geq 0$, we have $S_{n+1}-S_n \leq  1-(B_{n+1}-B_n)$. Hence, we get that almost surely for every $n\geq 0$,

\begin{equation}\label{Lembounds}
S_n -R_n \leq B_n \leq S_n,
\end{equation} where for every $n\geq 0$,

$$R_n:=\sum_{k\in \Xi_n}{(1-(B_{k+1}-B_k))}.$$ Then, if $(r_n)_n$ is a sequence of i.i.d. random variables with the same law as $\mathcal{R}_{r}$, we have for every $n\geq 0$ that
$$R_n \overset{(d)}{=} \sum_{k=1}^{\xi_n}{(1+r_k)}.$$ 

In order to prove that $(B_n)_{n\geq 0}$ has the same scaling limit as $(S_n)_{n\geq 0}$, the idea is to establish that the process $(R_n)_{n\geq 0}$ is small compared to $(S_n)_{n\geq 0}$. 

Recall from (\ref{defsigma}) the definition of the stopping times $(\sigma_k)_{k\geq 0}$. In other words, $\sigma_k$ is the time of the $k^{\text{th}}$ return of $(F_n)_{n\geq 0}$ to zero, and $\sigma_{k-1}+1$ is the first time after $\sigma_{k-1}$ when $(F_n)_{n\geq 0}$ reaches $\mathbb{Z}_+$ (by construction). Notice that the total number of returns of the process to $0$ before time $n$ equals $\xi_n=\#\{k\geq 0 : \sigma_k < n\}$. The following lemma gives a bound on the asymptotic behaviour of this quantity.

\begin{Lem}\label{cvgprobaxi}For every $\alpha>0$, we have the convergence in probability 

$$\frac{\xi_n}{n^{1/3+\alpha}} \overset{\mathbb{P}}{\underset{n \rightarrow + \infty}{\longrightarrow}} 0.$$
\end{Lem}

\begin{proof}The strong Markov property applied to the stopping times $(\sigma_k)_{k\geq 0}$ yields that for every $k\geq 0$ and $n\geq 0$,

\begin{equation}\label{EqnXi}
  \mathbb{P}(\xi_n>k)\leq \prod_{j=1}^k{\mathbb{P}(\sigma_j-\sigma_{j-1}\leq n)}.
\end{equation} Now, we may simplify the computation with the following remark: for every $k\geq 1$, $\sigma_{k}-\sigma_{k-1}$ stochastically dominates the time needed by the random walk $(\hat{S}_n)_{n\geq 0}$ started from $0$ with steps distributed as $\hat{X}$ to reach $\mathbb{Z}_-\setminus\{0\}$. Indeed, due to the fact that $F_{\sigma_{k-1}+1}\geq 1$ almost surely we have (for instance by a coupling argument) that for every $k\geq 1$ and $n\geq 0$,

$$\mathbb{P}(\sigma_{k}-\sigma_{k-1}\leq n)\leq \mathbb{P}(\sigma'\leq n),$$ where $\sigma'$ is defined exactly as $\sigma=\sigma_1$, but for the Markov chain $(F_n)_{n\geq 0}$ with initial distribution $\delta_1$. 

Then, observe that when started at length $1$, the process $(F_n)_{n\geq 0}$ is bounded from below by a random walk with steps distributed as $\hat{X}$, also started at $1$ and up to the first entrance of this random walk into $\mathbb{Z}_-$ (using again a coupling argument which consists in not taking the positive part in (\ref{DefF})). In other words, for every $k\geq 1$ and $n\geq 0$,
$$\mathbb{P}(\sigma_{k}-\sigma_{k-1}\leq n)\leq \mathbb{P}(\bar{\sigma}\leq n),$$ where $\bar{\sigma}:=\inf\{n \geq 0 : \hat{S}_n+1 \leq 0\}$. Thus, we get from (\ref{EqnXi}) that $\mathbb{P}(\xi_n>k)\leq (\mathbb{P}(\bar{\sigma}\leq n))^k$, and now want to apply a result due to Rogozin that we recall for the sake of completeness.

\begin{Thbis}{(Theorem 0 in \cite{doney_exact_1982}, \cite{rogozin_distribution_1971})} Let $(X_k)_{k\geq 0}$ be a sequence of i.i.d. random variables, $S_0:=0$ and for every $n\geq 1$, $S_n:=\sum_{k=1}^n{X_k}$. Let also $\sigma_+:=\inf\left\lbrace n\geq 1 \mid S_n>0\right\rbrace$ be the first strict ascending ladder epoch of the random walk $(S_n)_{n\geq 0}$, and assume that 

\begin{equation}\label{AssumptionRogozin}
	\sum_{k=1}^{+\infty}k^{-1}P(S_k>0)=\sum_{k=1}^{+\infty}k^{-1}P(S_k\leq 0)=+\infty.
\end{equation} Then, the following assertions are equivalent:
\begin{enumerate}
	\item $n^{-1}\sum_{k=1}^n{\mathbb{P}(S_k>0)} \underset{n \rightarrow +\infty}{\longrightarrow} \gamma \in [0,1]$.
	\item There exists a slowly varying function $L_+$ such that $P(\sigma_+\geq n) \underset{n \rightarrow +\infty}{\sim} n^{-\gamma}L_+(n).$
\end{enumerate}

\end{Thbis}

Let us check the assumptions of this result. Using the convergence (\ref{CvgSkoShat}), we get

$$\mathbb{P}(\hat{S}_n>0)=\mathbb{P}\left(\frac{\hat{S}_n}{n^{2/3}}>0\right) \underset{n \rightarrow +\infty}{\longrightarrow} P_0(\kappa \mathcal{S}_1 >0)= \rho=\frac{2}{3},$$ where $\mathcal{S}$ is the spectrally negative $3/2$-stable process. It is thus clear that the identity (\ref{AssumptionRogozin}) is satisfied by the process $(-\hat{S}_n)_{n\geq 0}$. Moreover, Césaro's Lemma ensures that

$$\frac{1}{n}\sum_{k=1}^n{\mathbb{P}(-\hat{S}_k>0)} \underset{n \rightarrow +\infty}{\longrightarrow} \gamma=\frac{1}{3}.$$ Thus, the theorem applies and yields
\begin{equation}\label{equivalent}
  \mathbb{P}(\bar{\sigma}\geq n) \underset{n \rightarrow +\infty}{\sim} n^{-1/3}L_+(n),
\end{equation}
 where $L_+$ is a slowly varying function at infinity. In particular, for every $\alpha>0$ and $\varepsilon>0$, 

$$\mathbb{P}(\xi_n>n^{1/3+\alpha}\varepsilon)\leq (\mathbb{P}(\bar{\sigma}\leq n))^{n^{1/3+\alpha}\varepsilon}=(1-\mathbb{P}(\bar{\sigma}>n))^{n^{1/3+\alpha}\varepsilon} \underset{n \rightarrow +\infty}{\longrightarrow} 0,$$ as wanted.\end{proof}

Let us come back to the scaling limit of the process $(B_n)_{n\geq 0}$. We have for every $\alpha>0$ and $n$ large enough

\begin{align}\label{decomposition}
\frac{R_n}{n^{1/3+\alpha}}&\overset{(d)}{=} \frac{1}{n^{1/3+\alpha}}\sum_{k=1}^{\xi_n}{(1+r_k)}=\frac{\xi_n}{n^{1/3+\alpha}}\frac{1}{\xi_n}\sum_{k=1}^{\xi_n}{(1+r_k)}.\end{align} We observed that all the $(\sigma_k)_{k\geq 0}$ are almost surely finite, and thus the quantity $\xi_n$ goes to infinity almost surely with $n$. Using the previous lemma and the law of large numbers, we get that

$$\frac{R_n}{n^{2/3}} \overset{\mathbb{P}}{\underset{n \rightarrow + \infty}{\longrightarrow}} 0,$$ and then
$$\left(\frac{R_{\lfloor\lambda^{3/2}t\rfloor}}{\lambda}\right)_{t\geq 0} \overset{\mathbb{P}}{\underset{\lambda \rightarrow + \infty}{\longrightarrow}} 0$$ for the topology of uniform convergence on compact sets, and thus for Skorokhod's topology. Recalling the inequality (\ref{Lembounds}) and the convergence (\ref{CvgSkoS}), we immediately have

$$\left(\frac{B_{\lfloor \lambda t\rfloor}}{\lambda^{2/3}}\right)_{t \geq 0} \underset{\lambda \rightarrow +\infty}{\overset{(d)}{\longrightarrow}} \kappa(\mathcal{S}_t)_{t \geq 0},$$ and thus the proof of Proposition \ref{cvgscallim} with an obvious substitution, using that $\lambda^{-1}B_0 \underset{\lambda \rightarrow +\infty}{\longrightarrow} a$ to identify the starting point of the stable process.

\subsubsection{Stopped peeling process}\label{SectionStoppedPeeling}

We now focus on the possible situations when the process $(B_n)_{n \geq 0}$ reaches $\mathbb{Z}_-$, which happens almost surely since $(B_n)_{n\geq 0}$ is bounded from above by the recurrent random walk $(S_n)_{n\geq 0}$ we previously introduced. Throughout this section, we denote by $T$ the first time when $(B_n)_{n\geq 0}$ reaches nonpositive values: $$T:=\inf\{n \geq 0 : B_n \leq 0 \}.$$ We now split our study into two cases, depending on the value $\vert  B_{T} \vert$ of the so-called \textit{overshoot} of the process $(B_n)_{n\geq 0}$ at its first entrance into $\mathbb{Z}_-$. We implicitly exclude the events $\{\vert  B_{T} \vert=0\}$ and $\{\vert  B_{T} \vert=\lfloor \lambda b \rfloor\}$, whose probability will be proved to vanish when $\lambda$ goes to infinity in Section \ref{SectionProofThmCPBondCase}.

\vspace{5mm}

\textbf{Case 1: $\vert  B_{T} \vert<\lfloor \lambda b \rfloor$.} In this situation, only a fraction of the finite white segment on the boundary is swallowed, see Figure \ref{Case1bond}. Notice that the swallowed finite black segment has size $B_{T-1}$, since it is the black segment just before the last face is revealed, i.e. at time $T-1$. Unless explicitly mentioned, we now work conditionally on the event $\{\vert  B_{T} \vert<\lfloor \lambda b \rfloor\}$.

\begin{figure}[h]
\begin{center}
\includegraphics[scale=1]{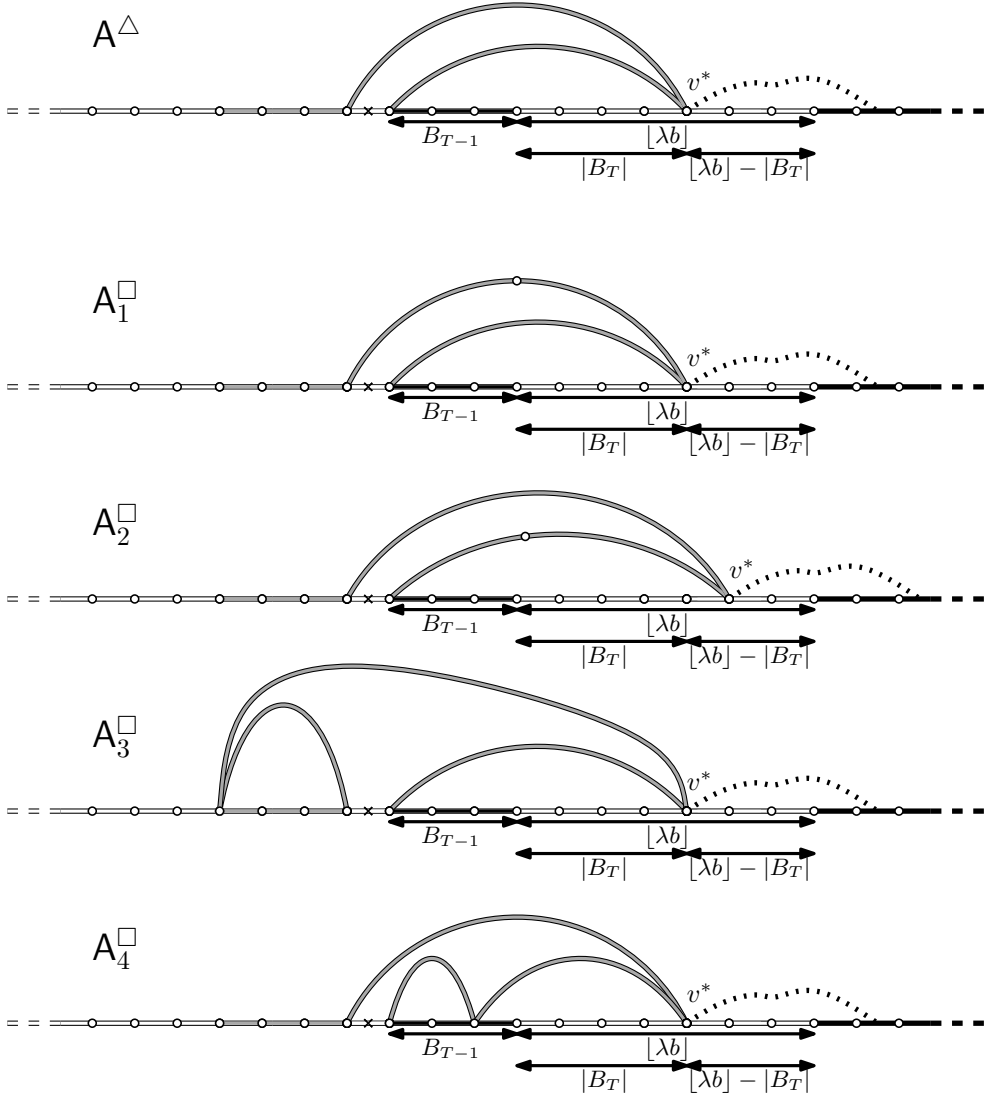}
\end{center}
\caption{Configurations when $\vert  B_{T} \vert<\lfloor \lambda b \rfloor$.}
\label{Case1bond}
\end{figure}

\vspace{3mm}

In each case, the crossing event $\mathsf{C}^*(\lambda a, \lambda b)$ implies the existence of an open path linking the rightmost revealed vertex on the boundary, denoted by $v^*$, and the infinite black segment on the boundary (see the dashed path on Figure \ref{Case1bond}). Indeed, the exploration process has the property that the open connected component of the initial finite black segment is entirely swallowed by the last face that is revealed. (For the crossing event to occur, one also needs that some bottom edges of this face are open, see Figure \ref{Case1bond}, but an upper bound is enough for our purpose.)

The situation may again be simplified using the spatial Markov property, because this event has the same probability as the following: consider a map which has the law of the UIHP-$*$, and let $\mathsf{C}^1_{\lambda}$ be the event that there exists a path linking the origin of the map and the infinite black segment of the boundary, where the boundary of the map has the colouring of Figure \ref{Crossing1} (the crossing event $\mathsf{C}^1_{\lambda}$ is again represented with a dashed path). Notice that the free segment on the boundary can have arbitrary but finite length (depending on the previous peeling steps) but this plays no role in our discussion.

\begin{figure}[h!]
\begin{center}
\includegraphics[scale=1.6]{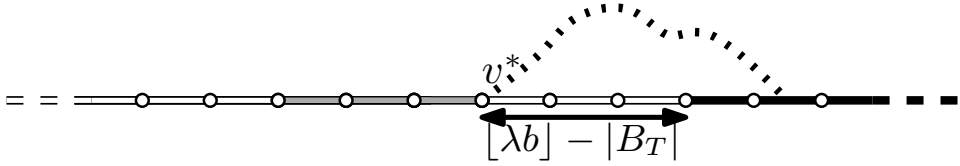}
\end{center}
\caption{The crossing event $\mathsf{C}^1_{\lambda}$.}
\label{Crossing1}
\end{figure}

\vspace{3mm}

Roughly speaking, the idea here is that when $\lambda$ becomes large, so does the length $d^{\lambda}:= \lfloor \lambda b \rfloor-\vert B_{T} \vert$ of the white segment that separates the origin from the black segment and the event $\mathsf{C}^1_{\lambda}$ becomes unlikely. To be precise, the following lemma holds.

\begin{Lem}\label{LemPrelimCase1}
Conditionally on $\vert  B_{T} \vert<\lfloor \lambda b \rfloor$, we have
$$\lfloor \lambda b \rfloor-\vert B_{T} \vert \overset{\mathbb{P}}{\underset{\lambda \rightarrow + \infty}{\longrightarrow}} +\infty.$$
\end{Lem}
\begin{proof}Recall that $B_0=\lfloor \lambda a \rfloor$ almost surely by definition and observe that if 

\begin{equation}\label{tildeT}
  \tilde{T}:=\inf\left\lbrace t\geq 0 : \lambda^{-1}  B_{\lfloor \lambda^{3/2} t\rfloor} \leq 0\right\rbrace,
\end{equation} then $\lfloor \lambda^{3/2} \tilde{T}\rfloor=T$. Thus, using Proposition \ref{cvgscallim}, we get that conditionally given $\vert B_{T} \vert<\lfloor \lambda b \rfloor$,

$$\lambda^{-1} B_{T}  \overset{(d)}{\underset{\lambda \rightarrow +\infty}{\longrightarrow}} \kappa \mathcal{S}_{\tau} \ \left( \textrm{given} \ \kappa\vert\mathcal{S}_{\tau}\vert<b\right),$$ where $\tau:= \inf\{t\geq 0 : \mathcal{S}_t \leq 0\}$ and $\mathcal{S}$ is the spectrally negative Lévy $3/2$-stable process started at $a$. As a consequence, we have that conditionally given $\vert B_{T} \vert<\lfloor \lambda b \rfloor$,
$$\lambda^{-1} (\lfloor \lambda b \rfloor-\vert B_{T}\vert) \overset{(d)}{\underset{\lambda \rightarrow +\infty}{\longrightarrow}} b-\kappa\mathcal{S}_{\tau} \ \left( \textrm{given} \ \kappa\vert\mathcal{S}_{\tau}\vert<b\right),$$ 
Now, following Proposition \ref{stableovershoot}, the distribution of the overshoot of the process $\mathcal{S}$ is absolutely continuous with respect to the Lebesgue measure on $\mathbb{R}_+$, which yields the expected result.\end{proof}

\vspace{5mm}

\textbf{Case 2: $\vert  B_{T} \vert>\lfloor \lambda b \rfloor$.} In this situation, the whole finite white segment of the boundary has been swallowed by the last revealed face, which divides the problem into different subcases shown in Figure \ref{conf2}. Unless explicitly mentioned, we now work conditionally on the event $\{\vert  B_{T} \vert>\lfloor \lambda b \rfloor\}$. 

\begin{figure}[h]
\begin{center}
\includegraphics[scale=1]{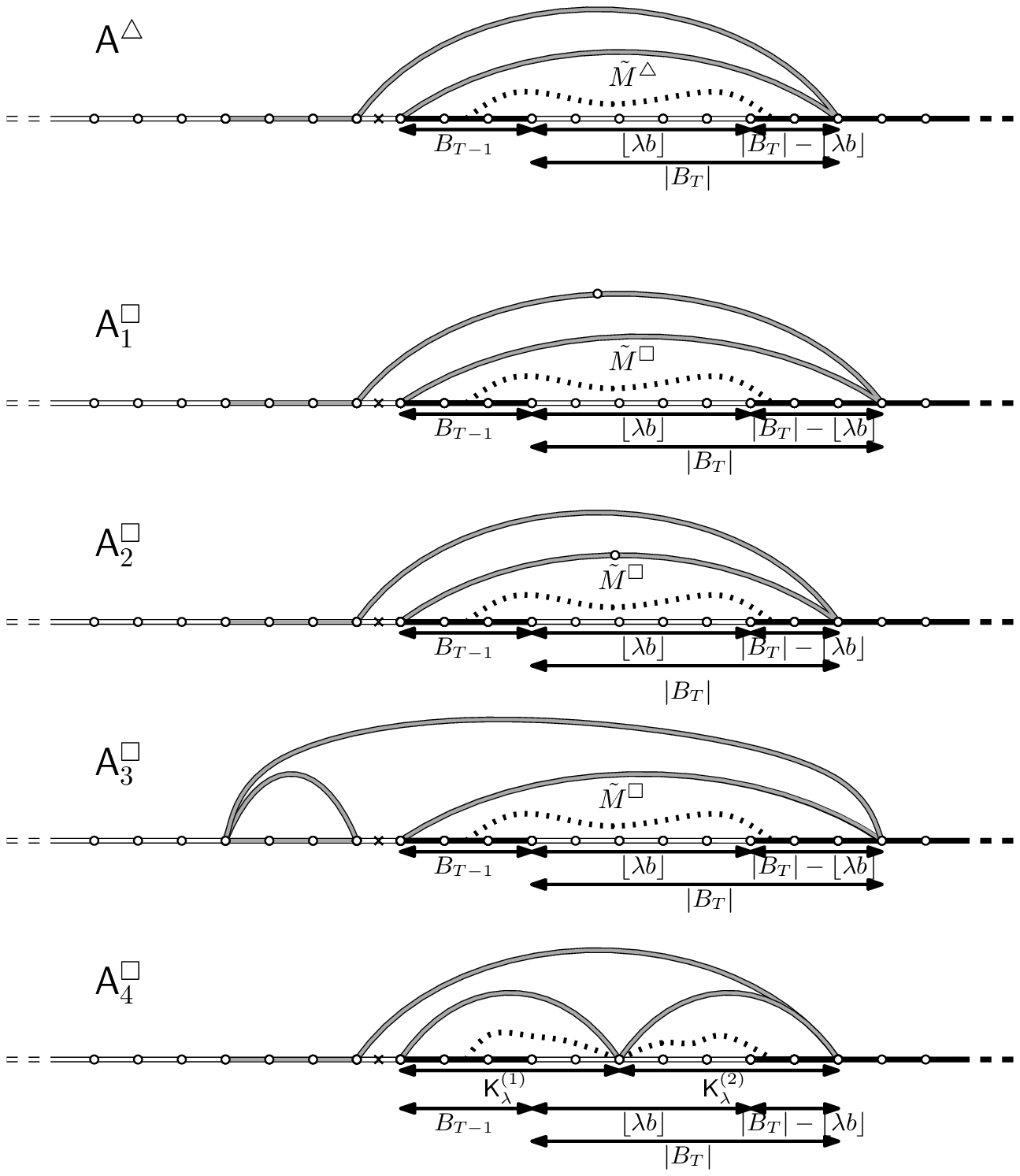}
\end{center}
\caption{Configurations when $\vert  B_{T} \vert>\lfloor \lambda b \rfloor$.}
\label{conf2}
\end{figure}

\vspace{3mm}

In the first four cases $(\mathsf{A}^{\triangle}-\mathsf{A}^{\square}_3)$, either the crossing event is realized by bottom edges of the last revealed face (which occurs with probability less than one), or it depends on a crossing event between black segments of the boundary in a submap $\tilde{M}^*$, which is a Boltzmann $*$-angulation of the $(B_{T-1} + \vert B_{T} \vert + 1)$-gon (or $(B_{T-1} + \vert B_{T} \vert + 2)$-gon, depending on the configuration). Those events are represented by a dashed path on Figure \ref{conf2}. The reason for this is that the finite black segment on the boundary at the second to last step of  the peeling process is in the same cluster as the initial one (those even share edges).

Roughly speaking, the idea here is that when $\lambda$ becomes large, this Boltzmann $*$-angulation converges in law towards the UIHP-$*$ with infinite black segments on the boundary, so that the above crossing event occurs almost surely asymptotically. The fifth case $\mathsf{A}^{\square}_4$ is slightly different, because the crossing represented by the dashed path seems unlikely. In fact, this fifth case does not occur asymptotically in a sense we will make precise. We will need the following analogue of Lemma \ref{LemPrelimCase1} concerning the size of the black segments that appear in the Boltzmann map.

\begin{Lem}\label{LemPrelimCase2}
Conditionally on $\vert  B_{T} \vert>\lfloor \lambda b \rfloor$, we have $$\vert B_{T} \vert-\lfloor \lambda b \rfloor \overset{\mathbb{P}}{\underset{\lambda \rightarrow + \infty}{\longrightarrow}} +\infty \quad \text{and} \quad B_{T-1} \overset{\mathbb{P}}{\underset{\lambda \rightarrow + \infty}{\longrightarrow}} +\infty,$$ 
and for every $\beta >1$, 

$$\lambda^{-\beta}(\vert B_{T} \vert-\lfloor \lambda b \rfloor) \overset{\mathbb{P}}{\underset{\lambda \rightarrow + \infty}{\longrightarrow}} 0 \quad\text{ and } \quad \lambda^{-\beta}B_{T-1} \overset{\mathbb{P}}{\underset{\lambda \rightarrow + \infty}{\longrightarrow}} 0.$$

\end{Lem}

\begin{proof} The proof is analogous to that of Lemma \ref{LemPrelimCase1}. Again, recall that $B_0=\lfloor \lambda a \rfloor$ almost surely by definition and that if $\tilde{T}:=\inf\{t\geq 0 : \lambda^{-1}  B_{\lfloor \lambda^{3/2} t\rfloor} \leq 0\}$, then $\lfloor \lambda^{3/2} \tilde{T}\rfloor=T$. Applying Proposition \ref{cvgscallim}, we get that conditionally given $\vert B_{T} \vert>\lfloor \lambda b \rfloor$,

\begin{equation}\label{CvgLawCouple}
  \lambda^{-1} (B_{T-1},B_{T}) \overset{(d)}{\underset{\lambda \rightarrow +\infty}{\longrightarrow}} \kappa (\mathcal{S}_{\tau-},\mathcal{S}_{\tau}) \ \left( \textrm{given} \ \kappa\vert\mathcal{S}_{\tau}\vert>b\right),
\end{equation} where $\tau:= \inf\{t\geq 0 : \mathcal{S}_t \leq 0\}$ and $\mathcal{S}$ is the spectrally negative Lévy $3/2$-stable process started at $a$. Therefore, we have that conditionally given $\vert B_{T} \vert>\lfloor \lambda b \rfloor$,
$$\lambda^{-1} (B_{T-1},\vert B_{T}\vert-\lfloor \lambda b \rfloor) \overset{(d)}{\underset{\lambda \rightarrow +\infty}{\longrightarrow}} (\kappa\mathcal{S}_{\tau-},\kappa\mathcal{S}_{\tau}-b) \ \left( \textrm{given} \ \kappa\vert\mathcal{S}_{\tau}\vert>b\right),$$ We conclude the proof exactly as before, using the absolute continuity of the joint distribution of the undershoot and the overshoot of the process $\mathcal{S}$ with respect to the Lebesgue measure, given by Proposition \ref{stableovershoot}.\end{proof}

\vspace{3mm}

We now treat the fifth configuration, denoted by $\mathsf{A}^{\square}_4$ in Figure \ref{conf2}, which occurs only in the quadrangular case. Precisely, we work on the event (still denoted $\mathsf{A}^{\square}_4$) that the last revealed quadrangle has all of its four vertices lying on the boundary of the map, two of them being on the right of the root edge. On this event, we can introduce $\mathsf{K}^{(1)}_{\lambda}$ and $\mathsf{K}^{(2)}_{\lambda}$, which are the (random) lengths of the two finite segments determined on the right boundary by the last revealed quadrangle (see Figure \ref{conf2} and a formal definition in (\ref{DefK})). 

\begin{Lem}\label{LemK1K2} Conditionally on $\vert  B_{T} \vert>\lfloor \lambda b \rfloor$ and on the event $\mathsf{A}^{\square}_4$, we have

$$\lambda^{-1}\min\left(\mathsf{K}^{(1)}_{\lambda},\mathsf{K}^{(2)}_{\lambda}\right)\overset{\mathbb{P}}{\underset{\lambda \rightarrow + \infty}{\longrightarrow}} 0.$$ As a consequence,

$$\max\left(B_{T-1}-\mathsf{K}^{(1)}_{\lambda}, \vert B_{T}\vert - \lfloor \lambda b \rfloor-\mathsf{K}^{(2)}_{\lambda}\right)\overset{\mathbb{P}}{\underset{\lambda \rightarrow + \infty}{\longrightarrow}} +\infty.$$
\end{Lem}

\begin{proof} Let us introduce the sequence of random variables $(K^{(1)}_n,K^{(2)}_n)_{n\geq 0}$, that describe the length of the two finite segments determined on the right boundary by the quadrangle revealed at step $n$ of the exploration process. We adopt the convention that $K^{(1)}_n=K^{(2)}_n=0$ if there are less than two segments determined on the right boundary. In particular, these random variables are i.i.d., and on the event $A^{\square}_4$, we have the identity

\begin{equation}\label{DefK}
	\left(\mathsf{K}^{(1)}_{\lambda},\mathsf{K}^{(2)}_{\lambda}\right):=\left(K^{(1)}_{T-1},K^{(2)}_{T-1}\right).
\end{equation} Now, for every $x>0$,

$$\mathbb{P}\left( K^{(1)}_0\geq x, K^{(2)}_0\geq x \right)=\sum_{k_1, k_2\geq x}{q^{\square}_{k_1,k_2}},$$ and $q^{\square}_{k_1,k_2}$ is equivalent to $\iota_{\square}k_1^{-5/2}k_2^{-5/2}$ when $k_1,k_2$ go to infinity, using the estimation (\ref{equivZ}). Thus, if $x$ is large enough, we have $q^{\square}_{k_1,k_2}\leq Ck_1^{-5/2}k_2^{-5/2}$ for every $k_1,k_2\geq x$, where $C$ is a positive constant. This yields

\begin{equation}\label{estimate}
	\mathbb{P}\left( K^{(1)}_0\geq x, K^{(2)}_0\geq x \right)\leq C\sum_{k_1, k_2\geq x}{k_1^{-5/2}k_2^{-5/2}}\leq C'x^{-3},
\end{equation} for another positive constant $C'$, provided that $x$ is large enough. Then, for every $\lambda>0$, $\varepsilon>0$ and $\delta >0$ we have

$$\mathbb{P}\left(K^{(1)}_{T-1}\geq \lambda \varepsilon ,K^{(2)}_{T-1}\geq \lambda \varepsilon \right)\leq \mathbb{P}\left( T > \lambda^{3/2+\delta} \right) + \mathbb{P}\left( \exists \ 0\leq n < \lambda^{3/2+\delta} \mid K^{(1)}_{n}\geq \lambda \varepsilon,K^{(2)}_{n}\geq \lambda \varepsilon \right).$$ For the first term, using the random variable $\tilde{T}:=\inf\{t\geq 0 : \lambda^{-1}  B_{\lfloor \lambda^{3/2} t\rfloor} \leq 0\}$ introduced in (\ref{tildeT}), we have $\lfloor \lambda^{3/2} \tilde{T}\rfloor=T$ and from Proposition \ref{cvgscallim}

$$\lambda^{-3/2}T \overset{(d)}{\underset{\lambda \rightarrow +\infty}{\longrightarrow}} \tau,$$ where $\tau:= \inf\{t\geq 0 : \mathcal{S}_t \leq 0\}$ and $\mathcal{S}$ is the spectrally negative Lévy $3/2$-stable process. The random variable $\tau$ being almost surely finite by standard properties of this process, we obtain that

$$\mathbb{P}\left( T > \lambda^{3/2+\delta} \right)\underset{\lambda \rightarrow +\infty}{\longrightarrow}0.$$ The second term can be controlled straightforwardly from (\ref{estimate}) using an union bound, giving

$$\mathbb{P}\left( \exists \ 0\leq n < \lambda^{3/2+\delta} \mid K^{(1)}_{n}\geq \lambda \varepsilon,K^{(2)}_{n}\geq \lambda \varepsilon \right)\leq \lambda^{3/2+\delta}C'(\lambda\varepsilon)^{-3}\underset{\lambda \rightarrow +\infty}{\longrightarrow}0.$$ We can now conclude that

\begin{equation}\label{Part1}
  \lambda^{-1}\min\left(K^{(1)}_{T-1},K^{(2)}_{T-1}\right)\overset{\mathbb{P}}{\underset{\lambda \rightarrow + \infty}{\longrightarrow}} 0, 
\end{equation} convergence which still holds conditionally on $\vert  B_{T} \vert>\lfloor \lambda b \rfloor$, observing that the probability of this event converges when $\lambda$ goes to infinity by Proposition \ref{stableovershoot} and the convergence in law of Proposition \ref{cvgscallim} (see a precise statement in (\ref{CvgProbaOvershoot})).

On the other hand, using the convergence in law (\ref{CvgLawCouple}), conditionally given $\vert  B_{T} \vert>\lfloor \lambda b \rfloor$ we get that

\begin{align*}
&\lambda^{-1}\max\left(B_{T-1}-K^{(1)}_{T-1}, \vert B_{T}\vert - \lfloor \lambda b \rfloor-K^{(2)}_{T-1}\right)\\
&\geq \lambda^{-1} \min\left(B_{T-1},\vert B_{T}\vert-\lfloor \lambda b \rfloor \right) -\lambda^{-1}\min\left(K^{(1)}_{T-1},K^{(2)}_{T-1}\right)\\
&\overset{(d)}{\underset{\lambda \rightarrow +\infty}{\longrightarrow}} \min\left(\kappa\mathcal{S}_{\tau-},\kappa\mathcal{S}_{\tau}-b\right) \ \left( \textrm{given} \ \kappa\vert\mathcal{S}_{\tau}\vert>b\right)
\end{align*}and thus by the very same argument as in Lemma \ref{LemPrelimCase2}, conditionally on $\vert  B_{T} \vert>\lfloor \lambda b \rfloor$,

\begin{equation}\label{Part2}
  \max\left(B_{T-1}-K^{(1)}_{T-1}, \vert B_{T}\vert - \lfloor \lambda b \rfloor-K^{(2)}_{T-1}\right)\overset{\mathbb{P}}{\underset{\lambda \rightarrow + \infty}{\longrightarrow}} +\infty.
\end{equation} Applying the identity (\ref{DefK}), the expected result follows from the convergences in probability (\ref{Part1}) and (\ref{Part2}), restricted on $\mathsf{A}^{\square}_4$. \end{proof}


Intuitively, the consequence of this lemma is that only two situations may occur on the event $\mathsf{A}^{\square}_4$: either $B_{T-1}-\mathsf{K}^{(1)}_{\lambda}$ or $\vert B_{T}\vert - \lfloor \lambda b \rfloor-\mathsf{K}^{(2)}_{\lambda}$ is large, or more precisely, respectively $\mathsf{K}^{(1)}_{\lambda}$ or $\mathsf{K}^{(2)}_{\lambda}$ is small with respect to $\lambda$. This corresponds to the subcases $\mathsf{A}^{\square}_{4,1}$ and $\mathsf{A}^{\square}_{4,2}$ of Figure \ref{K1K2}. Roughly speaking, this means that quadrangles ``look like" triangles when $\lambda$ is large.

\begin{figure}[h]
\begin{center}
\includegraphics[scale=1]{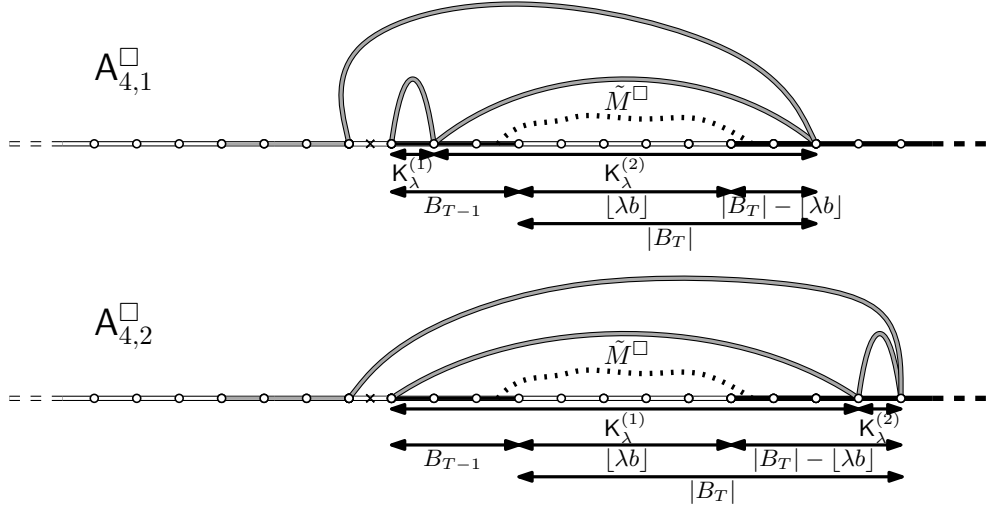}
\end{center}
\caption{The likely subcases on the event $\mathsf{A}^{\square}_4$.}
\label{K1K2}
\end{figure}

Let us be more precise here. The events $\mathsf{A}^{\square}_{4,1}$ and $\mathsf{A}^{\square}_{4,2}$ are defined as follows:

$$\mathsf{A}^{\square}_{4,1}:=\mathsf{A}^{\square}_4\cap\left\lbrace\mathsf{K}^{(1)}_{\lambda}<B_{T-1}\right\rbrace  \quad \text{ and } \quad \mathsf{A}^{\square}_{4,2}:=\mathsf{A}^{\square}_4\cap\left\lbrace\mathsf{K}^{(2)}_{\lambda}<\vert B_{T}\vert -\lfloor \lambda b \rfloor \right\rbrace.$$ Lemma \ref{LemK1K2} ensures that the probability of the event $\mathsf{A}^{\square}_4\cap\{\mathsf{K}^{(1)}_{\lambda}\geq B_{T-1},\mathsf{K}^{(2)}_{\lambda}\geq \vert B_{T}\vert -\lfloor \lambda b \rfloor\}$ vanishes when $\lambda$ becomes large, so that we can restrict our attention to the events $\mathsf{A}^{\square}_{4,1}$ and $\mathsf{A}^{\square}_{4,2}$ instead of $\mathsf{A}^{\square}_4$ in the next part, recalling that all the statements hold conditionally on the event $\{\vert  B_{T} \vert>\lfloor \lambda b \rfloor\}$.

Looking at the dashed paths on Figure \ref{K1K2} and using the above discussion, one gets that in every likely case, the crossing event $\mathsf{C}^{*}(\lambda a, \lambda b)$ depends on the simple crossing event in the submap $\tilde{M}^{*}$, which is always a Boltzmann $*$-angulation of a proper polygon thanks to the spatial Markov property.

The key quantities are $d_l^{\lambda}$ and $d_r^{\lambda}$, defined as the lengths of the two open segments on the boundary of that map (say with $d_r^{\lambda}$ the segment that is closer from the last edge we peeled), whose boundary condition is thus that represented on Figure \ref{Crossing2}. In the first four configurations $(\mathsf{A}^{\triangle}-\mathsf{A}^{\square}_3)$ described in Figure \ref{conf2}, we had $d_l^{\lambda}=B_{T-1}$ and $d_r^{\lambda}=\vert B_{T} \vert-\lfloor \lambda b \rfloor $. The situation is a bit more complicated in the cases of Figure \ref{K1K2}, but with the convention that $\mathsf{K}^{(1)}_{\lambda}=\mathsf{K}^{(2)}_{\lambda}=+\infty$ out of the event $\mathsf{A}^{\square}_4$, we have the identities

$$d_l^{\lambda}=B_{T-1}\boldsymbol{1}_{\left\lbrace\mathsf{K}^{(1)}_{\lambda}>B_{T-1}+\lfloor \lambda b\rfloor\right\rbrace}+\left(B_{T-1}-\mathsf{K}^{(1)}_{\lambda}\right)\boldsymbol{1}_{\left\lbrace\mathsf{K}^{(1)}_{\lambda}< B_{T-1}\right\rbrace},$$ and 
$$d_r^{\lambda}=(\vert B_{T} \vert-\lfloor \lambda b\rfloor)\boldsymbol{1}_{\left\lbrace\mathsf{K}^{(2)}_{\lambda}>\vert B_{T} \vert \right\rbrace}+\left(\vert B_{T} \vert-\lfloor \lambda b\rfloor-\mathsf{K}^{(2)}_{\lambda}\right)\boldsymbol{1}_{\left\lbrace\mathsf{K}^{(2)}_{\lambda}< \vert B_{T} \vert-\lfloor \lambda b\rfloor\right\rbrace}.$$ We now invoke Lemmas \ref{LemPrelimCase2} and \ref{LemK1K2}, together with the fact that $\mathsf{K}^{(1)}_{\lambda}+\mathsf{K}^{(2)}_{\lambda}=B_{T-1}+\vert B_{T}\vert$ on the event $\mathsf{A}^{\square}_4$, to get that (still conditionally on $\vert  B_{T} \vert>\lfloor \lambda b \rfloor$),

\begin{equation}\label{Cvginftyd}
  d_l^{\lambda} \overset{\mathbb{P}}{\underset{\lambda \rightarrow + \infty}{\longrightarrow}} +\infty \quad \text{ and } \quad d_r^{\lambda} \overset{\mathbb{P}}{\underset{\lambda \rightarrow + \infty}{\longrightarrow}} +\infty.
\end{equation} To sum up, we thus proved that conditionally on $\vert  B_{T} \vert>\lfloor \lambda b \rfloor$, the crossing event $\mathsf{C}^{*}(\lambda a, \lambda b)$ occurs almost surely when there is an open path between the two black segments on the boundary of the map of Figure \ref{Crossing2}, which is a Boltzmann $*$-angulation. In what follows, we denote by $\mathsf{C}^2_{\lambda}$ this event. Note that the length of the free segment on the boundary can either be $1$ or $2$ depending on the situation, and is thus uniformly bounded.

\begin{figure}[h]
\begin{center}
\includegraphics[scale=1]{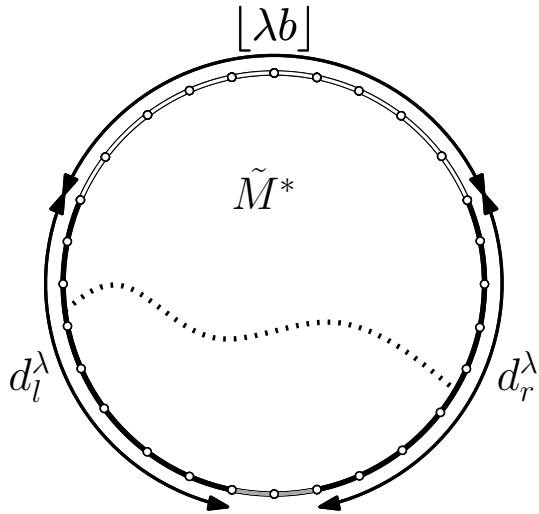}
\end{center}
\caption{The crossing event $\mathsf{C}^2_{\lambda}$.}
\label{Crossing2}
\end{figure}

\subsubsection{Asymptotic probabilities for submap crossings}

Throughout this section, we discuss the asymptotic probabilities of the events $\mathsf{C}^1_{\lambda}$ and $\mathsf{C}^2_{\lambda}$ we previously introduced. 

\vspace{5mm}

\textbf{1. The event $\mathsf{C}^1_{\lambda}$.} Recall that the event $\mathsf{C}^1_{\lambda}$ is the crossing event given by the existence of a path linking the origin of a UIHP-$*$ and the infinite open segment of the boundary, where the boundary of the map is coloured as in Figure \ref{Crossing1} (the free segment has arbitrary but finite length). We denote by $v^*$ the origin of the map, and recall that the length of the white segment located at the right of the origin satisfies from Lemma \ref{LemPrelimCase1} (we implicitly still work conditionally on $\vert  B_{T} \vert<\lfloor \lambda b \rfloor$) \begin{equation}\label{dInfty}
  d^{\lambda}:=\lfloor \lambda b \rfloor-\vert B_{T} \vert \overset{\mathbb{P}}{\underset{\lambda \rightarrow + \infty}{\longrightarrow}} +\infty.
\end{equation}For every $x\geq 0$, let $\partial^{(r)}_{x, \infty}$ be the set of vertices of the map that are located on the boundary, at least $x$ vertices away on the right of $v^*$. We are now able to compute the asymptotic probability we are interested in.

\begin{Prop}
\label{C1}
Conditionally on $\vert  B_{T} \vert<\lfloor \lambda b \rfloor$, we have
$$\mathbb{P}\left(\mathsf{C}_{\lambda}^1\right)\underset{\lambda \rightarrow + \infty}{\longrightarrow}0.$$ As a consequence, $\mathbb{P}(\mathsf{C}^{*}(\lambda a, \lambda b) \mid \vert  B_{T} \vert < \lfloor \lambda b \rfloor)\underset{\lambda \rightarrow + \infty}{\longrightarrow}0$.
\end{Prop}
\begin{proof}We first introduce the sequence of events $(\mathrm{V}_x)_{x\geq 0}$, where for every $x\geq 0$, $\mathrm{V}_x$ is the event that an open path links the origin $v^*$ of the map to the set of vertices $\partial^{(r)}_{x, \infty}$, without using edges of the right boundary. Then, we have that for every $\lambda>0$,
$$\mathbb{P}\left(\mathsf{C}_{\lambda}^1\right)= \mathbb{P}\left(\mathrm{V}_{d^{\lambda}}\right).$$ (We implicitly used that from any open path realizing $\mathsf{C}_{\lambda}^1$, we can extract one that does not use edges of the right boundary.) Now, $(\mathrm{V}_x)_{x\geq 0}$ is a decreasing sequence of events, so that for every $x\geq 0$ and $\lambda>0$,

$$\mathbb{P}\left(\mathsf{C}_{\lambda}^1\right)\leq \mathbb{P}\left(\mathrm{V}_{x}\right)+\mathbb{P}\left(d^{\lambda}\leq x \right).$$ Moreover, the limiting event

$$\mathrm{V}:=\bigcap_{x\geq 0}{\mathrm{V}_x}$$ is the event that infinitely many open paths from the origin intersect the right boundary (implicitly, at different vertices), again without using any edge of the right boundary. This event implies in particular that the cluster $\mathcal{C}$ of our map is still infinite if we set closed all the edges of the right boundary. Thus, the convergence (\ref{dInfty}) yields that

$$\limsup_{\lambda \rightarrow +\infty }\mathbb{P}\left(\mathsf{C}_{\lambda}^1\right)\leq\mathbb{P}\left(\mathrm{V}\right)\leq \Theta_{\rm{bond}}^*\left(p_{c,\rm{bond}}^*\right)=0.$$ The fact that there is no percolation at the critical point almost surely is straightforward from \cite{angel_percolations_2015}, Theorem 7. This result holds for a ``Free-White" boundary condition, and thus in our setting using the above remarks and a coupling argument. This ends the proof, using for the second statement the discussion of Section \ref{SectionStoppedPeeling}.\end{proof}

\textbf{2. The event $\mathsf{C}^2_{\lambda}$.} We now focus on the crossing event $\mathsf{C}^2_{\lambda}$, which is defined by the existence of an open path linking the two open segments on the boundary of a Boltzmann $*$-angulation, whose boundary is represented in Figure \ref{Crossing2}. Recall that the law of this map is denoted by $\mu_{v(\lambda)}^*$, where $v(\lambda)$ is the number of vertices on the boundary. The lengths of the open segments on the boundary are given by $d^{\lambda}_l$ and $d^{\lambda}_r$. For the sake of clarity, we may now assume that the free segment on the boundary has unit length, this single edge being chosen as the root of the map and denoted $w^*$ (the extension to the case with a free segment of length $2$ - or simply bounded - is immediate). For every $x\geq 0$, we let $\partial^{(l)}_{x}$ and $\partial^{(r)}_{x}$ be the sets of $x$ edges of the boundary of the map located immediately on the left and on the right of $w^*$. Finally, we set $d^{\lambda}:=\min(d^{\lambda}_l, d^{\lambda}_r)$. Using the previous arguments, in particular (\ref{Cvginftyd}) it is obvious that (again, we implicitly still work conditionally on $\vert  B_{T} \vert>\lfloor \lambda b \rfloor$)

\begin{equation}\label{Cvgces}
  v(\lambda)\underset{\lambda \rightarrow + \infty}{\longrightarrow} +\infty \quad \text{and} \quad  d^{\lambda}\overset{\mathbb{P}}{\underset{\lambda \rightarrow + \infty}{\longrightarrow}}+\infty.
\end{equation} In the next part, we will consider dual bond percolation on planar maps (i.e. bond percolation on the dual map), with the convention that the colour of a dual edge is the colour of the unique primal edge it crosses. We denote by $\Theta^*_{\rm{bond}'}$ and $p^*_{c,\rm{bond'}}$ the dual bond percolation probability and the dual bond percolation threshold on the UIHP-$*$. We can now prove that the event $\mathsf{C}_{\lambda}^2$ occurs asymptotically almost surely.

\begin{Prop}
\label{C2}
Conditionally on $\vert  B_{T} \vert>\lfloor \lambda b \rfloor$, we have $$\mathbb{P}\left(\mathsf{C}_{\lambda}^2\right)\underset{\lambda \rightarrow + \infty}{\longrightarrow} 1.$$ As a consequence, $\mathbb{P}(\mathsf{C}^{*}(\lambda a, \lambda b) \mid \vert  B_{T} \vert > \lfloor \lambda b \rfloor)\underset{\lambda \rightarrow + \infty}{\longrightarrow}1$.
\end{Prop}

\begin{proof}Let us first introduce the sequence of events $(\mathsf{V}_x)_{x\geq 0}$, where for every $x\geq 0$, $\mathsf{V}_x$ is the event that an open path links a vertex of $\partial^{(l)}_{x}$ to a vertex of $\partial^{(r)}_{x}$, without using edges of the boundary of the map (except possibly the root edge). Then, we have

$$\mathbb{P}\left(\mathsf{C}_{\lambda}^2\right)\geq \mathbb{P}\left(\mathsf{C}_{\lambda}^2,d^{\lambda}>x\right) \geq \mathbb{P}\left(\mathsf{V}_x,d^{\lambda}>x\right)\geq\mathbb{P}\left(\mathsf{V}_x\right)-\mathbb{P}\left(d^{\lambda}\leq x\right).$$ We emphasize that the quantity $\mathbb{P}\left(\mathsf{V}_x\right)$ still depends on $\lambda$ through the law of the map we consider, whose boundary has size $v(\lambda)$. We now want to restrict the events $(\mathsf{V}_x)_{x\geq 0}$ to a finite subset of the map. To do so, we introduce for every $x\geq 0$ and $R\geq 0$ the event $\mathsf{V}_{x,R}$, which is defined exactly as $\mathsf{V}_x$ with the additional condition that the crossing event occurs in the ball of radius $R$ of the map for the graph distance (and from the origin). We still have for every $x\geq 0$ and $R\geq 0$ that

$$\mathbb{P}\left(\mathsf{C}_{\lambda}^2\right)\geq \mathbb{P}\left(\mathsf{V}_{x,R}\right)-\mathbb{P}\left(d^{\lambda}\leq x\right).$$ Now, the event $\mathsf{V}_{x,R}$ is by construction measurable with respect to the ball of radius $R$ of the map, so that the convergence in law $\mu^*_{v(\lambda)} \underset{\lambda \rightarrow +\infty}{\Longrightarrow}\nu^*_{\infty,\infty}$ for the local topology implies that for every fixed $x\geq 0$ and $R\geq 0$,

$$\mathbb{P}\left(\mathsf{V}_{x,R}\right)\underset{\lambda \rightarrow + \infty}{\longrightarrow}\nu^*_{\infty,\infty}\left(\mathsf{V}_{x,R}\right).$$ Using also the convergence (\ref{Cvgces}), we get 

$$\liminf_{\lambda \rightarrow +\infty}{\mathbb{P}\left(\mathsf{C}_{\lambda}^2\right)}\geq \nu^*_{\infty,\infty}\left(\mathsf{V}_{x,R}\right),$$ and from the increasing property of the sequence $(\mathsf{V}_{x,R})_{R\geq 0}$ for every fixed $x\geq 0$, by letting $R$ go to infinity,

$$\liminf_{\lambda \rightarrow +\infty}{\mathbb{P}\left(\mathsf{C}_{\lambda}^2\right)}\geq \nu^*_{\infty,\infty}\left(\mathsf{V}_{x}\right).$$ At this point, let us introduce the complementary $\partial^{\infty}_x$ of the sets of edges $\partial^{(l)}_{x}$ and $\partial^{(r)}_{x}$ on the boundary (i.e. the edges that are at least $x$ edges away from the root on the boundary) and notice that almost surely, the complementary event $(\mathsf{V}_{x})^c$ is exactly the existence of a dual closed crossing between the dual of the root edge and dual edges of the set $\partial^{\infty}_x$ that do not use dual edges of $\partial^{(l)}_{x}\cup\partial^{(r)}_{x}$, using standard properties of dual bond percolation on planar maps. Up to extraction, we can assume that such a path do not uses dual edges of the boundary at all (out of its endpoints).

Following the same idea as in Proposition \ref{C1}, we introduce the decreasing sequence of events $(\mathsf{V}'_x)_{x\geq 0}$, where for every $x\geq 0$, $\mathsf{V}'_x$ is precisely the event that a closed dual path links the dual edge of the root to the set $\partial^{\infty}_{x}$ (without using other dual edges of the boundary than its endpoints). Thus, for every $x\geq 0$,

$$\limsup_{\lambda \rightarrow +\infty}{\mathbb{P}\left(\left(\mathsf{C}_{\lambda}^2\right)^c\right)}\leq \nu^*_{\infty,\infty}\left(\mathsf{V}'_{x}\right).$$ Again, the limiting event

$$\mathsf{V}':=\bigcap_{x\geq 0}{\mathsf{V}'_x}$$ is the event that infinitely many dual paths from the dual of the root edge intersect the boundary (at different edges), without using dual edges of the boundary out of their endpoints. This event implies in particular that the dual closed percolation cluster $\mathcal{C}'$ is infinite, even if we do not allow dual paths to use boundary edges - or equivalently if we suppose that the boundary is totally open, except the root edge which is free. Thus, 

$$\limsup_{\lambda \rightarrow +\infty}{\mathbb{P}\left(\left(\mathsf{C}_{\lambda}^2\right)^c\right)}\leq \nu^*_{\infty,\infty}\left(\mathsf{V}'\right)\leq \Theta^*_{\rm{bond}'}(1-p_{c,\rm{bond}}^*)=0.$$ Indeed, the parameter for dual bond percolation equals $1-p_{c,\rm{bond}}^*$ because we work at criticality, and we can invoke the fact that $p^*_{c,\rm{bond'}}=1-p^*_{c,\rm{bond}}$ and that there is no percolation at the critical point (see Section 3.4.1 in \cite{angel_percolations_2015} for details). One should notice here that this result holds for a ``Free-Black" boundary as proved in \cite{angel_percolations_2015}, but we easily conclude by the standard coupling argument and the above discussion. This ends the proof, using again for the second statement the discussion of Section \ref{SectionStoppedPeeling}.\end{proof}

\subsubsection{Proof of Theorem \ref{TheoremCP}: bond percolation case}\label{SectionProofThmCPBondCase}

We now prove the main result of this section. Recall that $\mathsf{C}^*(\lambda a, \lambda b)$ denotes the crossing event between the two black segments on the boundary in a UIHP-$*$ with the boundary condition of Figure \ref{colouring}. 

First, using again the stopping time $\tilde{T} :=\inf\{t\geq 0 : \lambda^{-1}  B_{\lfloor \lambda^{3/2} t\rfloor} \leq 0\}$, we get the convergence

\begin{equation*}
  \lambda^{-1}B_{\lfloor\lambda^{3/2}\tilde{T}\rfloor} \overset{(d)}{\underset{\lambda \rightarrow +\infty}{\longrightarrow}} \kappa \mathcal{S}_{\tau},
\end{equation*} where $\tau := \inf\{t\geq 0 : \mathcal{S}_t \leq 0\}$. Then, the distribution of the overshoot of the process $\mathcal{S}$ given in Proposition \ref{stableovershoot}, which is absolutely continuous with respect to the Lebesgue measure, implies on the one hand that
\begin{equation}\label{CvgProbaOvershoot}
  \mathbb{P}(\vert  B_{T} \vert > \lfloor \lambda b \rfloor)\underset{\lambda \rightarrow +\infty}{\longrightarrow}P_{\frac{a}{\kappa}}(\kappa\vert\mathcal{S}_{\tau}\vert> b)=\frac{1}{\pi}\arccos\left(\frac{b-a}{a+b}\right),
\end{equation} and on the other hand that the probability of the events $\{\vert  B_{T} \vert=0\}$ and $\{\vert  B_{T} \vert=\lfloor \lambda b \rfloor\}$ vanish when $\lambda$ becomes large. We can thus split up the problem conditionally on the value of $\vert  B_{T} \vert $, namely the crossing probability $\mathbb{P}(\mathsf{C}^{*}(\lambda a, \lambda b))$ has the same limit when $\lambda$ goes to infinity as the quantity 

\begin{align*}
\mathbb{P}(\mathsf{C}^{*}(\lambda a, \lambda b) \mid \vert  B_{T} \vert < \lfloor \lambda b \rfloor)\mathbb{P}(\vert  B_{T} \vert < \lfloor \lambda b \rfloor)+&\mathbb{P}(\mathsf{C}^{*}(\lambda a, \lambda b) \mid \vert  B_{T} \vert > \lfloor \lambda b \rfloor)\mathbb{P}(\vert  B_{T} \vert > \lfloor \lambda b \rfloor).
\end{align*}  The results of Propositions \ref{C1} and \ref{C2} thus yield that

$$\lim_{\lambda \rightarrow +\infty}{\mathbb{P}\left(\mathsf{C}^{\triangle}_{\rm{bond}}(\lambda a, \lambda b)\right)} = \lim_{\lambda \rightarrow +\infty}{\mathbb{P}\left(\mathsf{C}^{\square}_{\rm{bond}}(\lambda a, \lambda b)\right)}= \frac{1}{\pi}\arccos\left(\frac{b-a}{a+b}\right),$$ which is exactly Theorem \ref{TheoremCP} in the bond percolation case on the UIHPT and the UIHPQ. 

\subsection{Crossing probabilities for face percolation}

We now consider the crossing events for face percolation. We will see that although the peeling process is slightly different, the same method as for bond percolation applies in order to compute the scaling limits we are interested in. Recall that two faces are adjacent in this model if they share an edge. We still work conditionally on a ``White-Black-White-Black" boundary condition. Here, this means that the crossing event $\mathsf{C}^*_{\rm{face}}(\lambda a,\lambda b)$ has to occur between the two black segments of the boundary, such that one can imagine faces lying on the other side of the boundary as in Figure \ref{colouringFace}. 

\begin{figure}[h]
\begin{center}
\includegraphics[scale=1.6]{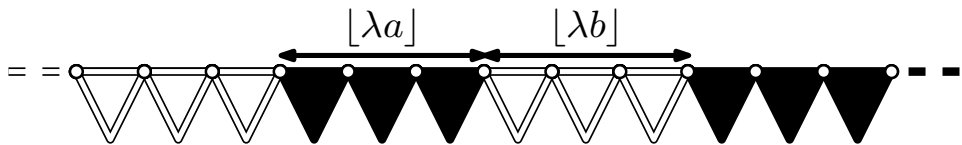}
\end{center}
\caption{The boundary condition for the face percolation problem.}
\label{colouringFace}
\end{figure}
The arguments being very similar to the previous section, we do not give as many details, except for the remarkable differences.

\subsubsection{Peeling process and scaling limit}

The main argument is again to define the appropriate peeling process, which is here simpler because we do not have to deal with a free segment on the boundary. However, a subtlety arising in this context is that the right thing to do is to peel the leftmost black edge instead of the rightmost white. The reason is that this exploration follows the rightmost percolation interface between the open and closed clusters, which is the correct one as far as the crossing event is concerned.

\begin{Alg} Consider a half-plane map which has the law of the UIHP-$*$ with a ``White-Black-White-Black" boundary condition (the outmost segments being infinite).

\vspace{2mm}

Reveal the face incident to the leftmost black edge of the boundary, including its colour.

\vspace{2mm}

After each step, repeat the algorithm on the UIHP-$*$ given by the unique infinite connected component of the current map, deprived of the revealed face. The exposed edges of the initial map form the boundary of the infinite connected component, whose colour is inherited from the adjacent revealed faces.\end{Alg} Note that this algorithm never ends because there is always a leftmost black edge on the boundary, although we are only interested in cases where the finite black segment has positive length.

The properties of this algorithm are the same as before, due to the spatial Markov property. In particular, the random variables $(\mathcal{E}_{n},\mathcal{R}_{r,n},c_n)_{n\geq 0}$, defined exactly as in Section \ref{SectionBondCase}, are i.i.d. and have the law of $\mathcal{E},\mathcal{R}_r$ and $c$ (which has the Bernoulli law of parameter $p^*_{c,\text{face}}$) respectively, where $c_n$ denotes the colour of the face revealed at step $n$ of the process and the other notations are the same as before. We again need the process $(B_n)_{n\geq 0}$, defined as follows. 

We let $B_0=\lfloor \lambda a \rfloor$ and for every $n\geq 0$:

\begin{equation}
  B_{n+1}=B_n+\mathbf{1}_{\{c_n=1\}}\mathcal{E}_n-\mathcal{R}_{r,n}-1.
\end{equation} The process $(B_n)_{n\geq 0}$ is a Markov chain with respect to the canonical filtration of the exploration process. Moreover, if we denote by $T:=\inf\{n \geq 0 : B_n \leq 0 \}$ the first entrance time of $(B_n)_{n\geq 0}$ in $\mathbb{Z}_-$, then for every $0\leq n < T$, $B_n$ is the \textbf{length of the finite black segment} at step $n$ of the exploration process. As long as $n<T$, it still holds that all the (black) faces incident to this segment are part of the open cluster of the initial finite black segment, due to the fact that our algorithm follows the rightmost percolation interface, while for $n\geq T$, the initial finite black segment has been swallowed.

\vspace{3mm}

\textbf{The process $(B_n)_{n\geq 0}$.} The properties of the peeling process yields that $(B_n)_{n \geq 0}$ is a random walk with steps distributed as $X$ defined by

$$X:=\mathbf{1}_{c=1}\mathcal{E}-\mathcal{R}_{r}-1.$$In particular, using the same arguments as before, we get

\begin{align*}
\mathbb{E}(X)&=p^*_{c,\text{face}}\mathbb{E}(\mathcal{E})-\mathbb{E}(\mathcal{R}_r)-1=
\left\lbrace
\begin{array}{ccc}
\frac{5}{3}p^{\triangle}_{c,\text{face}}-\frac{4}{3}  \  (*=\triangle) \\
\\
2p^{\square}_{c,\text{face}}-\frac{3}{2}  \  (*=\square) \\
\end{array}\right..
\end{align*} Thus, we have $\mathbb{E}(X)=0$ in both the UIHPT and the UIHPQ, and this yields the convergence $$\left(\frac{B_{\lfloor \lambda t\rfloor}}{\lambda^{2/3}}\right)_{t \geq 0} \underset{\lambda \rightarrow +\infty}{\overset{(d)}{\longrightarrow}} \kappa(\mathcal{S}_t)_{t \geq 0}$$ in the sense of Skorokhod, where $\mathcal{S}$ is the Lévy spectrally negative $3/2$-stable process started at $a$.

\subsubsection{Stopped peeling process and asymptotics}

Let us focus on the first time when $(B_n)_{n\geq 0}$ reaches $\mathbb{Z}_-$. Again, two cases are likely to happen:

\vspace{3mm}

\textbf{Case 1: $\vert  B_{T} \vert<\lfloor \lambda b \rfloor$.} In this situation, a finite part of the white segment on the boundary is swallowed. If the last revealed face is white, the crossing event cannot occur, while if it is black, the crossing event implies the existence of an open path linking the rightmost edge of the finite black segment and the infinite black segment on the boundary. This event, still denoted by $\mathsf{C}^1_{\lambda}$, is similar to that of the bond case. We again have that $$\lfloor \lambda b \rfloor-\vert B_{T} \vert \overset{\mathbb{P}}{\underset{\lambda \rightarrow + \infty}{\longrightarrow}} +\infty,$$ so that the crossing event $\mathsf{C}^1_{\lambda}$ has probability bounded by the percolation probability in this model and we get

$$\lim_{\lambda \rightarrow + \infty}{\mathbb{P}\left(\mathsf{C}_{\lambda}^1\right)} \leq \Theta^*_{\rm{face}}(p_{c,\rm{face}}^*)=0,$$ using Theorem 6 of \cite{angel_percolations_2015}.

\vspace{3mm}

\textbf{Case 2: $\vert  B_{T} \vert>\lfloor \lambda b \rfloor$.} In this situation, the whole white segment on the boundary is swallowed. We have $$\vert B_{T} \vert-\lfloor \lambda b \rfloor \overset{\mathbb{P}}{\underset{\lambda \rightarrow + \infty}{\longrightarrow}} +\infty \quad \text{and} \quad  B_{T-1} \overset{\mathbb{P}}{\underset{\lambda \rightarrow + \infty}{\longrightarrow}} +\infty,$$ as in the bond percolation case. Whether the last revealed face is open or closed, the crossing event is implied by a crossing event in a Boltzmann map, whose boundary has the colouring of Figure \ref{Crossing2}. We denote by $\mathsf{C}^2_{\lambda}$ this event, and we have that $$d_l^{\lambda} \overset{\mathbb{P}}{\underset{\lambda \rightarrow + \infty}{\longrightarrow}} +\infty\quad \text{and} \quad d_r^{\lambda} \overset{\mathbb{P}}{\underset{\lambda \rightarrow + \infty}{\longrightarrow}} +\infty$$ in every likely case. Then, when $\lambda$ goes to infinity, we have that $(\mathsf{C}^2_{\lambda})^c$ implies the closed dual face percolation event - dual face percolation being face percolation where adjacent faces share a vertex, or equivalently site percolation on the dual map with the same adjacency rules. This is analogous to the star-lattice in the case of $\mathbb{Z}^d$, and we have that $1-p_{c,\rm{face}'}^*=p_{c,\rm{face}}^*$, with no percolation at the critical point, see Section 3.4.2 in \cite{angel_percolations_2015} (with a fully open boundary condition). As a consequence, we get:

$$\lim_{\lambda \rightarrow + \infty}{\mathbb{P}\left(\left(\mathsf{C}_{\lambda}^2\right)^c\right)} \leq \Theta^*_{\rm{face}'}(1-p_{c,\rm{face}}^*)=0.$$ 

The statements we obtained up to this point yield, using the same arguments as in Section \ref{SectionBondCase}, that

$$\lim_{\lambda \rightarrow +\infty}{\mathbb{P}\left(\mathsf{C}^{\triangle}_{\rm{face}}(\lambda a, \lambda b)\right)} = \lim_{\lambda \rightarrow +\infty}{\mathbb{P}\left(\mathsf{C}^{\square}_{\rm{face}}(\lambda a, \lambda b)\right)}= \frac{1}{\pi}\arccos\left(\frac{b-a}{a+b}\right),$$ which is exactly Theorem \ref{TheoremCP} in the case of face percolation on the UIHPT and the UIHPQ.

\subsection{Crossing probabilities for site percolation}

The last case we study is site percolation on the UIHPQ. We work conditionally on the same ``White-Black-White-Black" boundary condition, represented in Figure \ref{colouringSite}.

\begin{figure}[h]
\begin{center}
\includegraphics[scale=1.6]{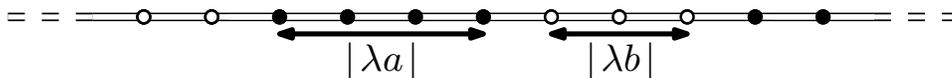}
\end{center}
\caption{The boundary condition for the site percolation problem.}
\label{colouringSite}
\end{figure}

We are still interested in the scaling limit of the probability of the crossing event that the two black segments are part of the same percolation cluster, denoted by $\mathsf{C}^{\square}_{\rm{site}}(\lambda a,\lambda b)$, at the critical point $p_{c,\rm{site}}^{\square}$.

\subsubsection{Peeling process and scaling limit}

Here, the appropriate peeling process follows the same idea as in the bond percolation case, but we use the vertex-peeling process (Algorithm \ref{VertexPeelingProcess}) introduced in Section \ref{SectionPCQ} in order to peel vertices instead of edges. Recall that a ``White-(Free)-Any" boundary condition means that we have an infinite white segment on the left, and then (possibly) a finite free segment, with no specific condition on the right part. Again, the right part will turn out to be coloured ``Black-White-Black" due to the boundary condition we are interested in, but we define the algorithm in the general setting.

\begin{Alg}Consider a half-plane map which has the law of the UIHPQ with a ``White-(Free)-Any" boundary condition.

\begin{itemize}

\item Reveal the colour of the rightmost free vertex, if any:

\begin{itemize}
\item If it is black, repeat the algorithm.

\item If it is white, mark this vertex and execute the vertex-peeling process, without revealing the colour of the vertices.

\end{itemize} 

\item If there is no free vertex on the boundary of the map, as in the first step, mark the rightmost white vertex of the infinite white segment and execute the vertex-peeling process, without revealing the colour of the vertices.
\end{itemize} After each step, repeat the algorithm on the unique infinite connected component of the current map, deprived of the revealed faces. The map we obtain is again implicitly rooted at the next edge we have to peel.

\end{Alg}

As usual, the algorithm is well defined in the sense that the pattern of the boundary (White-(Free)-Any) is preserved.

Following the strategy of Section \ref{SectionBondCase}, we now define the two processes $(F_n)_{n\geq 0}$ and $(B_n)_{n\geq 0}$. Let us first recall some familiar notation. For every $n\geq 0$, we denote by $c_n$ the color of the revealed vertex at step $n$ of the exploration, with the convention that $c_n= 0$ when there is no free vertex on the boundary. When $c_n=0$, the vertex-peeling process is executed and we let $\sigma_n+1$ be the number of steps of this process (so that $\sigma_n\geq 0$). Then, for every $0\leq i \leq \sigma_n$, we denote by $E^{(n)}_i$, respectively $R^{(l,n)}_i$, $R^{(r,n)}_i$ the number of exposed edges, respectively swallowed edges on the left and right of the root edge at step $i$ of the vertex-peeling process (in particular, $R^{(r,n)}_i=0$ for $i<\sigma_n$). All of those variables are set to zero when $c_n=1$ by convention.

Thanks to the spatial Markov property and the definition of the vertex-peeling process, if we restrict ourselves to the consecutive stopping times when $c_n=0$, the random variables $(\sigma_n)_{n \geq 0}$ are i.i.d., with geometric law of parameter $\mathbb{P}(\mathcal{R}_r=0)$: they have the same law as $\sigma$, defined for every $k\geq 0$ by $\mathbb{P}(\sigma=k)=(\mathbb{P}(\mathcal{R}_r=0))^k(1-\mathbb{P}(\mathcal{R}_r=0))$.

Conditionally on $(\sigma_n)_{n\geq 0}$, firstly the random variables $(E^{(n)}_i)_{n\geq 0,0\leq i\leq \sigma_n-1}$ and $(R^{(l,n)}_i)_{n\geq 0,0\leq i\leq \sigma_n-1}$ are i.i.d. and have the law of $\mathcal{E}$ and $\mathcal{R}_l$ conditionally on $\mathcal{R}_r=0$ respectively, and secondly, $(E^{(n)}_{\sigma_n})_{n\geq 0}$, $(R^{(l,n)}_{\sigma_n})_{n\geq 0}$ and $(R^{(r,n)}_{\sigma_n})_{n\geq 0}$ are i.i.d. and have the law of $\mathcal{E}$,$\mathcal{R}_l$ and $\mathcal{R}_r$ conditionally on $\mathcal{R}_r>0$ respectively. Finally, all those variables independent for different values of $n$. 

In what follows, we let $(\tilde{E}_i)_{i\geq 0}$, $(\tilde{R}_{l,i})_{i\geq 0}$, $\bar{E}$, $\bar{R}_{l}$ and $\bar{R}_{r}$  be independent random variables having the law of $\mathcal{E}$, $\mathcal{R}_l$ and $\mathcal{R}_r$ conditionally on $\mathcal{R}_r=0$ and $\mathcal{R}_r>0$ respectively. We also define $c$ that has the Bernoulli law of parameter $p_{c,\text{site}}^{\square}$. We can now define the processes $(F_n)_{n\geq 0}$ and $(B_n)_{n\geq 0}$.

\vspace{3mm}

First, for every $n\geq 0$, we let $F_n$ be the \textbf{length of the free segment} on the boundary at step $n$ of the peeling process. 

In this setting, this quantity is harder to study because of the vertex-peeling process that we (possibly) execute. Let us provide an alternative description of the process $(F_n)_{n\geq 0}$. We first have $F_0=0$. Then, let $F_0^{(0)}=0$ and for every $n\geq 0$, set inductively $F^{(n)}_0:=(F_n-1)_+$. We let for every $n\geq 0$ and $0\leq i \leq \sigma_n$,

\begin{equation*}
F^{(n)}_{i+1}:=
\left\lbrace
\begin{array}{ccc}
F^{(n)}_{i}+E^{(n)}_i-R^{(l,n)}_i-1 & \mbox{if} & F^{(n)}_{i}-R^{(l,n)}_i-1\geq 0 \\
E^{(n)}_i-1 & \mbox{otherwise} &\\
\end{array}\right..
\end{equation*} On the event $\{c_n=0\}$, this quantity describes the evolution of the length of the finite free segment during the execution of the vertex-peeling process. Then, we have for every $n\geq 0$,

\begin{equation}
F_{n+1}=\left\lbrace
\begin{array}{ccc}
F_n-1 & \mbox{if} & c_n=1\\
F^{(n)}_{\sigma_n+1} & \mbox{if} & c_n=0\\
\end{array}\right..
\end{equation} The process $(F_n)_{n\geq 0}$ is a Markov chain with respect to the canonical filtration of the exploration process, and $c_n= 0$ when  $F_n=0$.

\begin{Rk}Contrary to the bond percolation case, the process $(F_n)_{n\geq 0}$ can reach zero even when $c_n=0$, and can also take value zero at consecutive times with positive probability.
\end{Rk}

Then, we also let $B_0=\lfloor \lambda a \rfloor$ and for every $n\geq 0$:

\begin{equation}
B_{n+1}=\left\lbrace
\begin{array}{ccc}
B_n+1 & \mbox{if} & c_n=1\\
B_n+1-R^{(r,n)}_{\sigma_n} & \mbox{if} & c_n=0\\
\end{array}\right..
\end{equation} The process $(B_n)_{n\geq 0}$ is a Markov chain with respect to the canonical filtration of the exploration process. Moreover, if we denote by $T:=\inf\{n \geq 0 : B_n \leq 0 \}$ the first entrance time of $(B_n)_{n\geq 0}$ in $\mathbb{Z}_-$, then for every $0\leq n < T$, $B_n$ is the \textbf{length of the finite black segment} at step $n$ of the exploration process. Again, while $n<T$, the vertices of this segment are part of the open cluster of the initial black segment, while this initial segment has been swallowed for $n\geq T$.

\vspace{3mm}

\textbf{The process $(F_n)_{n\geq 0}$.} Exactly as in Section \ref{SectionBondCase}, the process $(F_n)_{n\geq 0}$ is not a random walk, but has the same behaviour when it is far away from zero. More precisely, let us define $\hat{\sigma}_+:=\inf\{n \geq 0 : F_n>0\}$ and $\hat{\sigma}_-:=\inf\lbrace n \geq \hat{\sigma}_+ : \exists \ 0\leq i\leq \sigma_n : F^{(n)}_{i}-R^{(l,n)}_i-1 < 0 \rbrace$, which is the first time after $\hat{\sigma}_+$ the finite free segment is swallowed (possibly during the vertex-peeling process). Note that $F_{\hat{\sigma}_-}$ is not zero in general, but if $F_n=0$ for $n\geq \hat{\sigma}_+$, then $\hat{\sigma}_-\leq n$. By construction, as long as $\hat{\sigma}_+\leq n< \hat{\sigma}_-$, we have

$$F_{n+1}=F_n-1+\mathbf{1}_{\{c_n=0\}}\left(\sum_{i=0}^{\sigma_n} \left(E_i^{(n)}-R_i^{(l,n)}-1\right)\right),$$ and $(F_{\hat{\sigma}_++n}-F_{\hat{\sigma}_+})_{0\leq n \leq \hat{\sigma}_--\hat{\sigma}_+}$ is a killed random walk with steps distributed as the random variable $\hat{X}$ defined by

$$\hat{X}:=\mathbf{1}_{\{c=0\}}\left(\sum_{i=0}^{\sigma-1} (\tilde{E}_i-\tilde{R}_{l,i}-1)+\bar{E}-\bar{R}_{l}-1\right)-1.$$ We get from the definitions that

\begin{align*}
\mathbb{E}(\hat{X})&=(1-p_{c,\rm{site}}^{\square})[\mathbb{E}(\sigma)\mathbb{E}(\tilde{E}_0-\tilde{R}_{l,0}-1)+\mathbb{E}(\bar{E}-\bar{R}_{l}-1)]-1\\
&=\frac{4}{9}\left(\frac{7}{2}\left[\mathbb{E}(\tilde{E}_0)-\mathbb{E}(\tilde{R}_{l,0})-1\right]+9-\frac{7}{2}\mathbb{E}(\tilde{E}_0)-\frac{9}{4}+\frac{7}{2}\mathbb{E}(\tilde{R}_{l,0})-1\right)-1=0,
\end{align*} using computations similar to that of Proposition \ref{pcquad} and the associated remark. This implies in particular that $\hat{\sigma}_-$ is almost surely finite (since it is bounded from above by the first return time of the process to zero). It is again easy to check that the random variable $\hat{X}$  satisfies the assumptions of Proposition \ref{stablerw}, so that if we denote by $(\hat{S}_n)_{n \geq 0}$ a random walk with steps distributed as $\hat{X}$, we have

\begin{equation}\label{CvgSkoShatSite}
\left(\frac{\hat{S}_{\lfloor \lambda t\rfloor}}{\lambda^{2/3}}\right)_{t \geq 0} \underset{\lambda \rightarrow +\infty}{\overset{(d)}{\longrightarrow}} \kappa(\mathcal{S}_t)_{t \geq 0},
\end{equation} in the sense of convergence in law for Skorokhod's topology, where $\mathcal{S}$ is the Lévy $3/2$-stable process.

\vspace{3mm}

\textbf{The process $(B_n)_{n\geq 0}$.} As before, $(B_n)_{n\geq 0}$ is not a random walk but behaves the same way when $F_n$ is not equal to zero. More precisely, let $\sigma^{(0)}_0=0$, $\sigma^{(1)}_0:=\inf\{n\geq 0 : F_n>0 \}$ and recursively for every $k\geq 1$,

$$\sigma^{(0)}_k:=\inf\{n\geq \sigma^{(1)}_{k-1} : F_n=0 \}\quad \text{ and } \quad \sigma^{(1)}_k:=\inf\{n\geq \sigma^{(0)}_k : F_n>0 \}.$$ From the fact that the vertex-peeling process exposes a positive number of vertices with positive probability, we know that $\sigma^{(1)}_0$ is almost surely finite. Moreover, using the above description of the process $(F_n)_{n\geq 0}$, it holds that started from any initial length, the free segment is swallowed in time $\hat{\sigma}_-$ which is also finite. This possibly happens during the vertex-peeling process, but in this case, with probability bounded away from zero, the vertex-peeling process ends at the next step with no exposed vertices and $F_{\hat{\sigma}_-}=0$. From there, we get that $\sigma^{(0)}_1$, and thus all the stopping times $(\sigma^{(0)}_k)_{k\geq 0}$ and $(\sigma^{(1)}_k)_{k\geq 0}$ are almost surely finite.

Using the strong Markov property and the definition of $(B_n)_{n\geq 0}$, for every $k\geq 0$, the process $(B_{\sigma^{(1)}_k+n}-B_{\sigma^{(1)}_k})_{0 \leq n \leq \sigma^{(0)}_{k+1}-\sigma^{(1)}_k}$ is a killed random walk with steps distributed as the random variable $X$ defined by 

$$X:=\mathbf{1}_{\{c=1\}}-\mathbf{1}_{\{c=0\}}(\bar{R}_r-1).$$In particular, we have $\mathbb{E}(X)=p_{c,\rm{site}}^{\square}-(1-p_{c,\rm{site}}^{\square})(\mathbb{E}(\mathcal{R}_r\mid \mathcal{R}_r>0)-1)=0$ and properties of the random variable $X$ yields that if $(S_n)_{n \geq 0}$ is a random walk with steps distributed as $X$, then

\begin{equation}\label{CvgSkoSSite}
\left(\frac{S_{\lfloor \lambda t\rfloor}}{\lambda^{2/3}}\right)_{t \geq 0} \underset{\lambda \rightarrow +\infty}{\overset{(d)}{\longrightarrow}} \kappa(\mathcal{S}_t)_{t \geq 0}\end{equation} in the sense of Skorokhod, where $\mathcal{S}$ is the Lévy $3/2$-stable process.

\vspace{3mm}

At this point, we are in the exact situation of Section \ref{SectionBondCase}, and the point is to get the following analogue of Proposition \ref{cvgscallim}.

\begin{Prop}\label{cvgscallimSite} We have, in the sense of convergence in law for Skorokhod's topology

$$\left(\frac{B_{\lfloor \lambda^{3/2} t\rfloor}}{\lambda}\right)_{t \geq 0} \underset{\lambda \rightarrow +\infty}{\overset{(d)}{\longrightarrow}}\kappa(\mathcal{S}_t)_{t \geq 0},$$ where $\mathcal{S}$ is the Lévy spectrally negative $3/2$-stable process started at $a$.
\end{Prop}

The proof of this result is analogous to that of Proposition \ref{cvgscallim}, but some notable differences exist as we now explain. Let again $S_0=B_0$ and introduce a sequence $(\beta_n)_{n \geq 0}$ of i.i.d. random variables with Bernoulli law of parameter $p^{\square}_{c,\text{site}}$. For every $n\geq 0$, we set

\begin{align*}
S_{n+1}
=S_n+\left\lbrace
\begin{array}{ccc}
B_{n+1}-B_n=\mathbf{1}_{\{c_n=1\}}-\mathbf{1}_{\{c_n=0\}}(R^{(r,n)}_{\sigma_n}-1) &\mbox{ when }  F_n>0 \\
\\
\beta_n+(1-\beta_n)(B_{n+1}-B_n)=\mathbf{1}_{\{\beta_n=1\}}-\mathbf{1}_{\{\beta_n=0\}}(R^{(r,n)}_{\sigma_n}-1)  &\mbox{ when }  F_n=0 \\
\end{array}\right..
\end{align*} 
Let also for every $n\geq 0$, $\Xi_n:=\{0\leq k < n : F_n=0\}$ and $\xi_n:=\#\Xi_n$. Using arguments similar to that of Section \ref{SectionBondCase}, we see that $(S_n)_{n\geq 0}$ is a random walk started from $B_0$ with steps distributed as the random variable $X$ we previously introduced, and that almost surely for every $n\geq 0$,

$$ S_n -R_n \leq B_n \leq S_n,$$ where $R_n:=\sum_{k\in \Xi_n}{(1-(B_{k+1}-B_k))}$. The only thing that remains to check is that for every $\alpha>0$,

\begin{equation}\label{cvgxisite}
\frac{\xi_n}{n^{1/3+\alpha}}\overset{\mathbb{P}}{\underset{n \rightarrow + \infty}{\longrightarrow}} 0.
\end{equation}

We will then use the arguments of the proof of Proposition \ref{cvgscallim} to conclude that $(R_n)_{n\geq 0}$ is small with respect to $(S_n)_{n\geq 0}$ and does not affect the scaling limit of the process $(B_n)_{n\geq 0}$. 

Here, the argument stands apart from the bond percolation case because the process $(F_n)_{n\geq 0}$ is slightly different from that of Section \ref{SectionBondCase} in the sense that it can stay at value zero for consecutive steps with positive probability (see the previous remark). However, if we introduce for $n\geq 0$ the quantity $\chi_n:=\#\{k\geq 0 : \sigma^{(0)}_k\leq n \}$, then the convergence (\ref{CvgSkoShatSite}) yield with the same proof as in Lemma \ref{cvgprobaxi} that

\begin{equation*}
\frac{\chi_n}{n^{1/3+\alpha}}\overset{\mathbb{P}}{\underset{n \rightarrow + \infty}{\longrightarrow}} 0.
\end{equation*} Indeed, for every $k\geq 0$, the random interval $\sigma^{(0)}_{k+1}-\sigma^{(1)}_k$ is still stochastically dominating the entrance time into $\mathbb{Z}_-$ of a random walk started at $1$, with steps distributed as $\hat{X}$. 

Moreover, thanks to the strong Markov property, the random variables $(\sigma^{(1)}_k-\sigma^{(0)}_k)_{k\geq 0}$ are i.i.d. with geometric law of parameter $\theta\in (0,1)$, where $\theta$ is the probability that the vertex-peeling process expose a positive number of vertices (for instance bounded from below by $\mathbb{P}(\mathcal{E}>1 \mid \mathcal{R}_r>0)>0$).  Then, for every $n\geq 0$, we have by construction

$$\xi_n\leq \sum_{k=0}^{\chi_n-1}{(\sigma^{(1)}_k-\sigma^{(0)}_k)},$$ and we get the expected convergence (\ref{cvgxisite}) using the very same identity as in (\ref{decomposition}), combined with the law of large numbers. This concludes the proof of Proposition \ref{cvgscallimSite}.

\subsubsection{Stopped peeling process and asymptotics}

Let us now focus on the situation when $(B_n)_{n\geq 0}$ reaches $\mathbb{Z}_-$. Two cases are likely to happen.

\vspace{3mm}

\textbf{Case 1: $\vert  B_{T} \vert<\lfloor \lambda b \rfloor$.} In that situation, a finite part of the white segment on the boundary is swallowed, see Figure \ref{CaseQ1}.

It is easy to see that for any possible configuration of the last revealed face, $\vert  B_T \vert<\lfloor \lambda b \rfloor$ implies that the percolation cluster of the finite black segment is confined in a finite region of the space, disconnected from the infinite one. Thus, the crossing event does not occur: $\mathbb{P}(\mathsf{C}^{\square}(\lambda a, \lambda b) \mid \vert B_{T} \vert < \lfloor \lambda b \rfloor)=0$.

\begin{figure}[h]
\begin{center}
\includegraphics[scale=1.6]{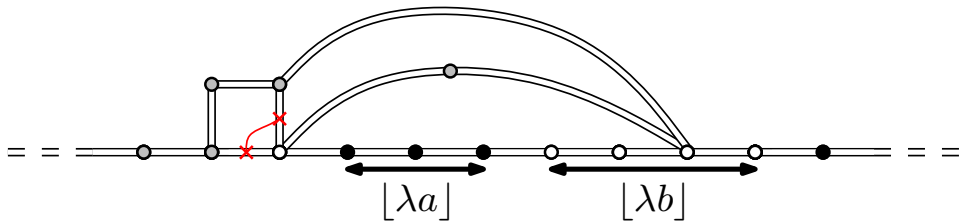}
\end{center}
\caption{A possible configuration when $\vert  B_T \vert<\lfloor \lambda b \rfloor$.}
\label{CaseQ1}
\end{figure}

\vspace{3mm}

\textbf{Case 2: $\vert  B_T \vert>\lfloor \lambda b \rfloor$.} In this situation, the whole white segment on the boundary is swallowed, see Figure \ref{CaseQ2}.

\begin{figure}[h]
\begin{center}
\includegraphics[scale=1.6]{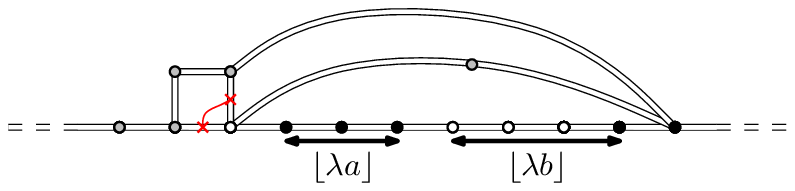}
\end{center}
\caption{The situation when $\vert B_T\vert>\lfloor \lambda b \rfloor$.}
\label{CaseQ2}
\end{figure}

First, we can treat the cases where four vertices of the last revealed face are lying on the boundary at the right of the origin vertex (the equivalent of the event $\mathsf{A}^{\square}_4$ of Section \ref{SectionBondCase}) exactly as before, and we get that the crossing event $\mathsf{C}^{\square}(\lambda a, \lambda b)$ is implied by the crossing event between the open segments on the boundary in a Boltzmann map with the boundary condition of Figure \ref{Crossingsite}. Again, the segment on the boundary may have size $1$ or $2$, but we suppose that it is made of a single white vertex which is the origin of the map $v^*$ for the sake of clarity (the extension being straightforward). We still denote by $\mathsf{C}^2_{\lambda}$ this crossing event. Recall that this map has law $\mu_{v(\lambda)}^{\square}$, where $v(\lambda)$ is the number of vertices on its boundary. 

For every $x\geq 0$, let $\partial^{(l)}_{x}$ and $\partial^{(r)}_{x}$ be the sets of $x$ vertices of the boundary of the map located immediately on the left and on the right of $v^*$. Finally, we can prove the same way as for bond percolation that the lengths $d^{\lambda}_l, d^{\lambda}_r$ of the finite black segments satisfy 
\begin{equation}\label{cvgProbadLambda}
d^{\lambda}:=\min(d^{\lambda}_l, d^{\lambda}_r) \overset{\mathbb{P}}{\underset{\lambda \rightarrow + \infty}{\longrightarrow}} +\infty.
\end{equation} 

\begin{figure}[h]
\begin{center}
\includegraphics[scale=1]{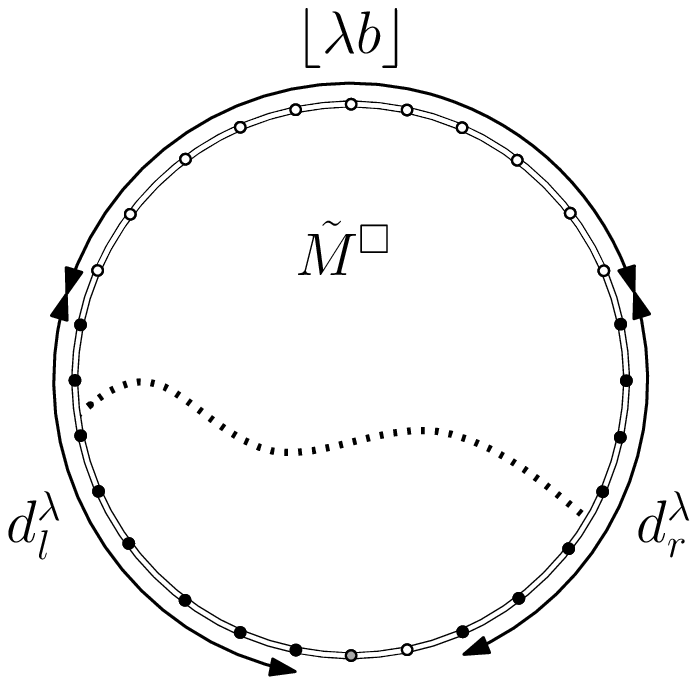}
\end{center}
\caption{The event $\mathsf{C}^2_{\lambda}$.}
\label{Crossingsite}
\end{figure}

The main problem is now to prove that the event $\mathsf{C}^2_{\lambda}$ occurs almost surely asymptotically. The proof is here very different from the bond percolation case, because we have no information on the dual percolation model. (More precisely, the analogue of the dual graph for site percolation should be the so-called \textit{matching graph}, which is in general non-planar.)

\begin{Prop}
\label{C2b}
Conditionally on $\vert  B_T \vert>\lfloor \lambda b \rfloor$, we have

$$\mathbb{P}\left(\mathsf{C}_{\lambda}^2\right)\underset{\lambda \rightarrow + \infty}{\longrightarrow} 1.$$As a consequence, $\mathbb{P}(\mathsf{C}^{*}(\lambda a, \lambda b) \mid \vert  B_{T} \vert > \lfloor \lambda b \rfloor)\underset{\lambda \rightarrow + \infty}{\longrightarrow}1$.\end{Prop}

\begin{proof}Recall the definition of the sequence of events $(\mathsf{V}_x)_{x\geq 0}$, where for every $x\geq 0$, $\mathsf{V}_x$ is the event that an open path links a vertex of $\partial^{(l)}_{x}$ to a vertex of $\partial^{(r)}_{x}$, without using vertices of the boundary of the map. Then, we have

$$\mathbb{P}\left(\mathsf{C}_{\lambda}^2\right)\geq \mathbb{P}\left(\mathsf{C}_{\lambda}^2,d^{\lambda}>x\right) \geq \mathbb{P}\left(\mathsf{V}_x,d^{\lambda}>x\right)\geq\mathbb{P}\left(\mathsf{V}_x\right)-\mathbb{P}\left(d^{\lambda}\leq x\right).$$ Again, the quantity $\mathbb{P}\left(\mathsf{V}_x\right)$ still depends on $\lambda$ through the law of the map we consider, whose boundary has size $v(\lambda)$. We use the same restriction of the events $(\mathsf{V}_x)_{x\geq 0}$ to the ball of radius $R$ of the map as in Proposition \ref{C2}, and with the same arguments, including the convergence (\ref{cvgProbadLambda}), we get
$$\liminf_{\lambda \rightarrow +\infty}{\mathbb{P}\left(\mathsf{C}_{\lambda}^2\right)}\geq \nu^*_{\infty,\infty}\left(\mathsf{V}_{x}\right).$$ Since the sequence of events $(\mathsf{V}_x)_{x\geq 0}$ is increasing, 

$$\liminf_{\lambda \rightarrow +\infty}{\mathbb{P}\left(\mathsf{C}_{\lambda}^2\right)}\geq \nu^*_{\infty,\infty}\left(\mathsf{V}\right).$$ Here, the limiting event

$$\mathsf{V}:=\bigcup_{x\geq 0}{\mathsf{V}_x}$$ is the event that there exists a crossing between the left and right boundary in the UIHPQ with a closed origin vertex. With the current definitions, this crossing cannot use vertices of the boundary. However, if it was the case, we could extract from the crossing a subpath which do not intersect the boundary. In other words, we can focus on the crossing event between the two infinite black segments in a UIHPQ with a totally open boundary condition, excluding the origin which is closed.

It is easy to see, using the standard coupling argument, that the event $\mathsf{V}$ has probability bounded from below by the probability of the very same event with an infinite free segment on the left, as in Figure \ref{colouringCinf2}.

\begin{figure}[h]
\begin{center}
\includegraphics[scale=1.6]{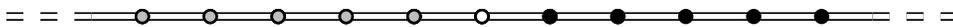}
\caption{The colouring of the boundary for the coupling argument.}
\label{colouringCinf2}
\end{center}
\end{figure}

The idea is now to use the same peeling process as in Section \ref{SectionPCQ}, \textit{with reversed colours}. It is defined as follows.

\vspace{3mm}

Reveal the colour of the rightmost free vertex on the boundary.
\begin{itemize}
\item If it is white, repeat the algorithm.
\item If it is black, mark this vertex and execute the vertex-peeling process. Then, repeat the algorithm on the unique infinite connected component of the map deprived of the faces revealed by the vertex-peeling process. 
\end{itemize} 
The algorithm ends when the initial finite white segment (here a single vertex) has been swallowed. 

We can check that the pattern of the boundary is preserved and that the steps of this process are i.i.d. exactly as before. For every $n\geq 0$, let $W_n$ be the length of the finite white segment at step $n$ of the exploration. Since $1-p_{c,\rm{site}}^{\square}<p_{c,\rm{site}}^{\square}$, the percolation model is subcritical for closed vertices and we have that the stopping time $T':=\inf\{n\geq 0 : W_n=0\}$ is almost surely finite, using the arguments of Section \ref{SectionPCQ}. Then, two cases may occur at time $T'$:

\begin{itemize}
\item There is an explicit black crossing (which occurs with positive probability), see Figure \ref{blackCrossing}.

\begin{figure}[h]
\begin{center}
\includegraphics[scale=1.6]{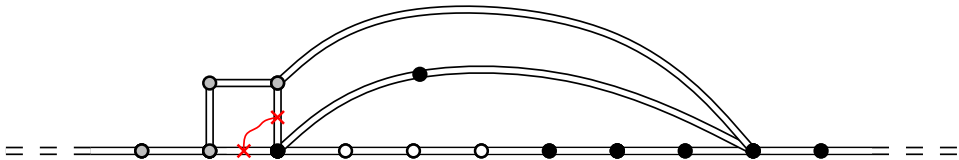}
\end{center}
\caption{A black crossing occurs.}
\label{blackCrossing}
\end{figure}

\item There is no black crossing, see Figure \ref{NoblackCrossing}.

\begin{figure}[h]
\begin{center}
\includegraphics[scale=1.6]{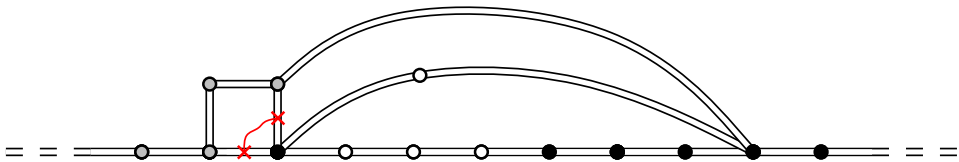}
\end{center}
\caption{No black crossing occurs.}
\label{NoblackCrossing}
\end{figure}

\end{itemize}
Let us focus on the case where there is no crossing. The total number $\mathsf{F}$ of (free) vertices discovered by the peeling process up to time $T'$ is  almost surely finite (because $T'$ also is), so that the probability of the crossing event we are interested in is now bounded from below by the probability of the same event in a planar map which has the law of the UIHPQ and the boundary condition of Figure \ref{situationNoCrossing}. Note that even though the $\mathsf{F}$ vertices discovered by the peeling process are free, it is helpful to suppose that they are closed vertices because they do not lie on the initial boundary of the map on which we want to exhibit a crossing event - in other words, the crossing now has to occur \textit{across} this finite closed segment.

\begin{figure}[h]
\begin{center}
\includegraphics[scale=1.6]{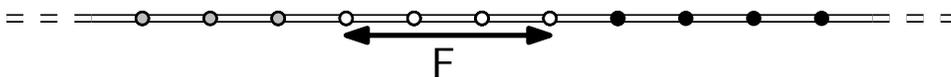}
\end{center}
\caption{The situation when no black crossing occurred.}
\label{situationNoCrossing}
\end{figure}

We now repeat the peeling process we previously introduced, with an initial white segment of size $\mathsf{F}$. Since $\mathsf{F}$ is almost surely finite, this process ends in finite time almost surely. The point is now to observe that the probability that a black crossing occurs at time $T'$ is bounded from below by $p^{\square}_{c,\text{site}}$, which is the probability that the free vertex between the two black vertices in the last revealed face (if any) is black. The successive executions of this process being independent, we end up with a black crossing in finite time almost surely, i.e. $\nu^*_{\infty,\infty}\left(\mathsf{V}\right)=1$, which ends the proof.\end{proof}

Using the same arguments as in Section \ref{SectionBondCase}, the statements we obtained yield that

$$\mathbb{P}\left(\mathsf{C}_{\rm{site}}^{\square}(\lambda a,\lambda b)\right) \underset{\lambda \rightarrow +\infty}{\longrightarrow} \frac{1}{\pi}\arccos\left(\frac{b-a}{a+b}\right),$$ which is exactly Theorem \ref{TheoremCP} in the case of site percolation on the UIHPQ, and thus ends the proof of the main result.

\bibliography{UniversalityPercolationRPM}
\bibliographystyle{abbrv}

\end{document}